\documentclass[12pt]{amsart}
\usepackage{amsmath}
\usepackage{hyperref}
\usepackage{amssymb}
\usepackage{latexsym}
\usepackage{amscd}
\usepackage{graphicx} 
\usepackage{amsthm}
\usepackage{mathrsfs}
\usepackage{xypic}
\usepackage{bm}
\usepackage{color}

\newdimen\AAdi%
\newbox\AAbo%
\def\AAk#1#2{\s_etbox\AAbo=\hbox{#2}\AAdi=\wd\AAbo\kern#1\AAdi{}}%
\def\AAr#1#2#3{\s_etbox\AAbo=\hbox{#2}\AAdi=\ht\AAbo\raise#1\AAdi\hbox{#3}}%
\font\tenmsb=msbm10 at 12pt \font\sevenmsb=msbm7 at 8pt
\font\fivemsb=msbm5 at 6pt
\newfam\msbfam
\textfont\msbfam=\tenmsb \scriptfont\msbfam=\sevenmsb
\scriptscriptfont\msbfam=\fivemsb
\def\Bbb#1{{\tenmsb\fam\msbfam#1}}
\newtheorem{thm}{Theorem}[section]
\newtheorem{lem}{Lemma}[section]
\newtheorem{cor}{Corollary}[section]

\newtheorem{pro}{Proposition}[section]
\newtheorem{defi}{Definition}[section]

\newtheorem{problem}{Problem}[section]

\newcommand{\ba}{\begin{array}}
\newcommand{\ea}{\end{array}}

\newcommand{\Section}[2]{\setcounter{equation}{0}
\allowdisplaybreaks
\section[#1]{#2}}

\def\n{\nabla}

\def\f#1#2{\frac{#1}{#2}}

\def\mc#1{\mathcal{#1}}

\def\pf#1{\frac{\partial}{\partial #1}}
\def\pd#1#2{\frac {\partial #1}{\partial #2}}

\def\td{\tilde}

\def\a{\alpha}
\def\be{\beta}

\def\p#1{\partial #1}

\def\de{\delta}

\def\ep{\varepsilon}

\def\G{\Gamma}
\def\g{\gamma}
\def\k{\kappa}

\def\th{\theta}

\def\si{\sigma}

\def\w{\wedge}

\def\R{\Bbb{R}}

\def\lan{\langle}
\def\ran{\rangle}
\def\ra{\rightarrow}

\def\ol{\overline}
\def\mb{\mathbf}
\def\O{\Bbb{O}}
\def\H{\Bbb{H}}
\def\Re{\text{Re }}
\def\Im{\text{Im }}
\def\Id{\mathbf{Id}}
\def\Arg{\text{Arg}}

\subjclass{58E20,53A10.}

\begin{document}
\title
[Submanifolds with constant Jordan angles] {Submanifolds with constant Jordan angles}

\author
[J. Jost, Y. L. Xin and Ling Yang]{J. Jost, Y. L. Xin and Ling Yang}
\address{Max Planck Institute for Mathematics in the
Sciences, Inselstr. 22, 04103 Leipzig, Germany.}
\email{jost@mis.mpg.de}
\address {Institute of Mathematics, Fudan University,
Shanghai 200433, China.} \email{ylxin@fudan.edu.cn}
\address{Institute of Mathematics, Fudan University,
Shanghai 200433, China.} \email{yanglingfd@fudan.edu.cn}
\thanks{The first author is supported by the ERC Advanced Grant
  FP7-267087. The second named author and the third named author are grateful to the Max Planck
Institute for Mathematics in the Sciences in Leipzig for its
hospitality and  continuous support. }

\begin{abstract}
To study the Lawson-Osserman's counterexample \cite{l-o} to the
Bernstein problem for minimal submanifolds of higher codimension, a
new geometric concept, submanifolds in Euclidean space with constant
Jordan angles(CJA), is introduced. By exploring
the second fundamental form of submanifolds with CJA, we can
characterize the Lawson-Osserman's cone  from the viewpoint of
Jordan angles.
\end{abstract}
\maketitle

\small \parskip0.1mm\tableofcontents \normalsize\parskip3mm

\Section{Introduction}{Introduction}\label{cja}

In previous works, we have systematically studied the Bernstein problem for complete minimal submanifolds of higher codimension in Euclidean space
 (see \cite{h-j-w,j-x,j-x-y2,j-x-y4,x-y1}). In particular, we could prove that a complete minimal submanifold in Euclidean space is affine linear
 if it does not deviate too much from a linear subspace in the sense that a certain function $v$ defined in terms of Jordan angles is bounded by 3.
 It is natural to ask whether that number is optimal. Now, there is the Lawson-Osserman's counterexample \cite{l-o} to the higher codimension Bernstein
 problem for which $v$ is identically 9. The aim of the present paper then is to understand this example in geometric terms, in particular in terms of
 Jordan angles. Here, the Jordan angles between two linear subspaces $P$ and $Q$ are the critical
values of the angle $\th$ between the  nonzero vectors $u$ in $P$
and their orthogonal projection $u^*$ in $Q$. When these Jordan
angles are constant for all the normal spaces of some submanifold
$M$ of Euclidean space and a fixed linear reference subspace, we say
that $M$ has constant Jordan angles. This is the fundamental concept
of our paper, and we abbreviate it as CJA. For a precise statement,
refer to Definition \ref{def-CJA} below. Now it turns out the
Lawson-Osserman's counterexample has CJA relative to the imaginary
quaternions when viewed as a subspace of the imaginary octonians.
Harvey-Lawson \cite{h-l} showed that the
Lawson-Osserman's cone is a four dimensional coassociative
submanifold in $\R^7$ which can be identified with the imaginary
octonians. Therefore, we study such coassociative submanifolds with
CJA and find that a coassociative graph with CJA relative to the
imaginary quaternions and at most two different normal Jordan angles
either is affine linear or a translate of a portion of the
Lawson-Osserman's cone.

For more precise statements, we now develop some notation and technical concepts.

\subsection{Jordan angles and angle spaces}\label{1.2}

Let $P$ and $Q_0$ be $m$-dimensional subspaces (i.e. $m$-planes) in $\R^{n+m}$.
The \textit{Jordan angles} between $P$ and $Q_0$ are the critical
values of the angle $\th$ between a nonzero vector $u$ in $P$ and its orthogonal projection $u^*$ in $Q_0$ as $u$ runs through $P$.
This concept was firstly introduced by Jordan \cite{j} in 1875, and they are also called
\textit{principal angles} in some references, e.g. \cite{d}. If $\th$
is a nonzero Jordan angle between $P$ and $Q_0$ determined by a unit vector $u$ in $P$ and its projection $u^*$ in $Q_0$, then $u$ is called
an \textit{angle direction} of $P$ relative to $Q_0$, and the $2$-plane spanned by $u$ and $u^*$ is called an \textit{angle 2-plane} between $P$ and $Q_0$
(see \cite{w}).

Denote by $\mc{P}_0$ the orthogonal projection of $\R^{n+m}$ onto $Q_0$ and by $\mc{P}$ the orthogonal projection of $\R^{n+m}$ onto $P$. Then for any $u\in P$ and $\ep\in Q_0$,
\begin{equation}
\aligned
\lan \mc{P}_0 u,\ep\ran&=\lan \mc{P}_0 u+(u-\mc{P}_0 u),\ep\ran=\lan u,\ep\ran\\
&=\lan u,\mc{P}\ep+(\ep-\mc{P}\ep)\ran=\lan u,\mc{P}\ep\ran
\endaligned
\end{equation}
and moreover
\begin{equation}
\lan (\mc{P}\circ\mc{P}_0)u,v\ran=\lan \mc{P}_0u,\mc{P}_0v\ran=\lan u,(\mc{P}\circ \mc{P}_0)v\ran
\end{equation}
holds for every $u,v\in P$, which implies $\mc{P}\circ \mc{P}_0$ is a nonnegative definite self-adjoint transformation
on $P$.

For any nonzero vector $u\in P$,
\begin{equation}
\cos^2 \angle(u,u^*)=\f{\lan u^*,u^*\ran}{\lan u,u\ran}=\f{\lan \mc{P}_0 u,\mc{P}_0 u\ran}{\lan u,u\ran}=\f{\lan (\mc{P}\circ \mc{P}_0)u,u\ran}{\lan u,u\ran}.
\end{equation}
Hence $\th$ is a Jordan angle between $P$ and $Q_0$ if and only if $\mu:=\cos^2\th$ is an eigenvalue of $\mc{P}\circ \mc{P}_0$, and
$u$ is an angle direction with respect to $\th$ if and only if $u$ is an eigenvector associated to the eigenvalue
$\mu$, i.e.
\begin{equation}
(\mc{P}\circ \mc{P}_0)u=\mu u=\cos^2\th\ u.
\end{equation}
 Therefore, all the angle directions with respect to $\th$ constitute a linear subspace of $P$, which is called
an \textit{angle space} of $P$ relative to $Q_0$ and we denote it by $P_\th$. In particular,
\begin{equation}
P_0=P\cap Q_0,\qquad P_{\pi/2}=P\cap Q_0^\bot.
\end{equation}
 The dimension of $P_\th$ is called the \textit{multiplicity} of
$\th$, which is denoted by $m_\th$. If we denote by $\text{Arg}(P,Q_0)$ the set consisting of all the Jordan angles between
$P$ and $Q_0$, then
\begin{equation}\label{oplus1}
P=\bigoplus_{\th\in \text{Arg}(P,Q_0)}P_\th
\end{equation}
and the angle spaces are mutually orthogonal to each other. Hence
\begin{equation}
m=\sum_{\th\in \text{Arg}(P,Q_0)}m_\th.
\end{equation}

The Jordan angles between two $m$-planes completely determine their relative positions. More precisely,
one can conclude that:

\begin{pro}\cite{w}\label{position}
Let $P_1,Q_1$ and $P_2,Q_2$ be any two pairs of $m$-planes in $\R^{n+m}$. If $\Arg(P_1,Q_1)=\Arg(P_2,Q_2)$
and the multiplicities of the corresponding Jordan angles are equivalent, then there exists a rigid motion
of $\R^{n+m}$, carrying $P_1, Q_1$ onto $P_2,Q_2$, respectively. And vice versa.
\end{pro}

Similarly, let  $\text{Arg}(Q_0,P)$ denote the set consisting of all the
Jordan angles between $Q_0$ and $P$, then
$\th\in \text{Arg}(Q_0,P)$ if and only
if $\mu:=\cos^2\th$ is an eigenvalue of $\mc{P}_0\circ \mc{P}$.
Denote by $(Q_0)_\th$ the angle space of $Q_0$ relative to $P$ associated to $\th$, then
$\ep\in (Q_0)_\th$ if and only if $(\mc{P}_0\circ
\mc{P})\ep=\cos^2\th\ \ep$, and
\begin{equation}\label{oplus2}
Q_0=\bigoplus_{\th\in \text{Arg}(Q_0,P)}(Q_0)_\th.
\end{equation}

Let $P^\bot$ and $Q_0^\bot$ be the orthogonal complements of $P$ and
$Q_0$, and denote by $\mc{P}^\bot$ and $\mc{P}_0^\bot$ the
orthogonal projections of $\R^{n+m}$ onto $P^\bot$ and $Q_0^\bot$,
respectively. As above, the set consisting of all the
Jordan angles between $P^\bot$ and $Q_0^\bot$ is denoted by
$\text{Arg}(P^\bot,Q_0^\bot)$, $P_\th^\bot$ denotes the angle space
associated to $\th\in \text{Arg}(P^\bot,Q_0^\bot)$, and
$m_\th^\bot:=\dim P_\th^\bot$ denotes the multiplicity of $\th$.

The following lemma
reveals the close relationship between $\text{Arg}(P,Q_0)$, $\Arg(Q_0,P)$ and
$\text{Arg}(P^\bot,Q_0^\bot)$.

\begin{lem}\label{Jordan}(\cite{j-x-y4})
Let $P,Q_0$ be $m$-planes in $\R^{n+m}$, then $\Arg(P,Q_0)=\Arg(Q_0,P)$ and the multiplicities of each corresponding Jordan angles are equivalent.
If we denote
\begin{equation}
R_\th:=P_\th+(Q_0)_\th
\end{equation}
for each $\th\in \Arg(P,Q_0)$,
then $R_\th\bot R_\si$ whenever $\th\neq \si$, and
\begin{equation}\label{Rth}
P+Q_0=\bigoplus_{\th\in \Arg(P,Q_0)}R_\th.
\end{equation}
For any $\th\in (0,\pi/2]$, $\th\in \Arg(P^\bot,Q_0^\bot)$ if and only if
 $\th\in \Arg(P,Q_0)$, and $m_\th^\bot=m_\th$, $R_\th=P_\th\oplus P_\th^\bot$. Moreover, for every $\th\in
\text{Arg}(P,Q_0)\cap (0,\pi/2)$, there exists an isometric
automorphism $\Phi_\th: R_\th\ra R_\th$, such that

(i) $\Phi_\th(P_\th)=P_\th^\bot$, $\Phi_\th(P_\th^\bot)=P_\th$;

(ii) $\Phi_\th^2=-\Id$;

(iii) For any nonzero vector $u\in P_\th$ ($v\in P_\th^\bot$),
$\Phi_\th(u)$ ($\Phi_\th(v)$) lies in the angle 2-plane generated by
$u$ ($v$); more precisely,
\begin{equation}\label{Phi}\aligned
\sec\th\ \mc{P}_0 u&=\cos\th\ u-\sin\th\ \Phi_\th(u),\\
\sec\th\ \mc{P}_0^\bot v&=\cos\th\ v-\sin\th\ \Phi_\th(v).
\endaligned
\end{equation}

\end{lem}

\noindent \textbf{Remark. }Let $P$ and $Q_0$ be a pair of intersecting planes in $\R^3$,
then $\Arg(P,Q_0)=\{\th,0\}$, where $\th$ is the dihedral angle between $P$ and $Q_0$.
$$\includegraphics[height=2in]{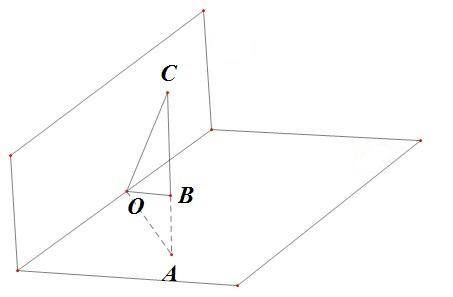}$$
Denote by $l$ their line of intersection and $O$ the origin of $\R^3$.
Choose $A\in \R^3$, such that $v:=\overrightarrow {OA}$ is a unit vector orthogonal to $P$. Through $A$, draw a perpendicular line to $Q_0$, intersecting
$Q_0$ at $B$, $P$ at $C$. Denote $u:=\f{\overrightarrow{OC}}{|\overrightarrow{OC}|}$, then $R_\th=\text{span}\{u,v\}$, $\Phi_\th(v)=u$ and $\Phi_\th(u)=-v$.
\bigskip

Denote
\begin{equation}
r(P):=\sum_{\th\in \text{Arg}(P,Q_0)\cap (0,\pi/2]}m_\th=\sum_{\th\in \text{Arg}(P^\bot,Q_0^\bot)\cap (0,\pi/2]}m_\th^\bot
\end{equation}
then $0\in \text{Arg}(P,Q_0)$ if and only if $r(P)<m$, and $m_0=m-r(P)$. Similarly
$0\in \text{Arg}(P^\bot,Q_0^\bot)$ if and only if $r(P)<n$, and $m_0^\bot=n-r(P)$.

\subsection{Angle space distributions and submanifolds with CJA}

Let $M$ be an $n$-dimensional submanifold in $\R^{n+m}$ and $Q_0$ be a fixed
$m$-plane in $\R^{n+m}$. Denote by $TM$ and $NM$
the tangent bundle and the normal bundle along $M$, respectively.
For any $p\in M$, denote by $\Arg(N_p M,Q_0)$ ($\Arg(T_p M,Q_0^\bot)$)  the set
consisting of all the Jordan angles between $N_p M$ ($T_p M$) and $Q_0$ ($Q_0^\bot$),
which are called \textit{normal (tangent) Jordan angles} at $p$.
Let $\th$ be a $[0,\pi/2]$-valued smooth function on $M$, if
$\th(p)\in \Arg(N_p M, Q_0)$ ($\th(p)\in \Arg(T_p M,Q_0^\bot)$),
we say $\th$ is a \textit{normal (tangent) Jordan angle function}  of $M$ relative to $Q_0$. Denote by $\Arg^N$ ($\Arg^T$)
the set consisting of all the normal (tangent) Jordan angle functions of
$M$ relative to $Q_0$.
If $\th$ is a smooth function on $M$ that is nonzero
everywhere, then Lemma \ref{Jordan} implies $\th\in \Arg^N$ if and only if $\th\in \Arg^T$.

Denote
\begin{equation}\aligned
N_\th M&:=\{\nu\in N_p M: p\in M, \nu\text{ is an angle direction associated to }\th(p)\},\\
T_\th M&:=\{v\in T_p M: p\in M, v\text{ is an angle direction associated to }\th(p)\}.
\endaligned
\end{equation}
Let $\mc{P}_0$ and $\mc{P}_0^\bot$ be orthogonal projections onto $Q_0$ and $Q_0^\bot$,
$(\cdot)^T$ and $(\cdot)^N$ denote orthogonal projections onto $T_p M$ and $N_p M$, respectively.
Then $\nu\in N_{p,\th}M:=N_\th M\cap N_p M$ if and only if
\begin{equation}
(\mc{P}_0 \nu)^N=\cos^2\th(p) \nu
\end{equation}
and similarly $u\in T_{p,\th}M:=T_\th M\cap T_p M$ if and only if
\begin{equation}
(\mc{P}_0^\bot u)^T=\cos^2\th(p) u.
\end{equation}
Let $m_\th^N(p):=\dim N_{p,\th} M$, $m_\th^T(p):=\dim T_{p,\th} M$ for every $p\in M$, then
$m_\th^N$ and $m_\th^T$ are both $\Bbb{Z}^+$-valued functions on $M$.

Based on \cite{no}, one can easily deduced that

\begin{lem}\label{dis}(\cite{j-x-y4})
Let $\th$ be a normal (tangent) Jordan angle function of $M$ relative to $Q_0$.
If $m_\th^N$ ($m_\th^T$) is a constant function on $M$, then $N_\th M$ ($T_\th M$) is a smooth subbundle of $NM$ ($TM$).
\end{lem}

In this case,
$N_\th M$ ($T_\th M$) is said to be a \textit{normal (tangent) angle space distribution} associated to $\th$.
A curve $\g: t\in (a,b)\mapsto \g(t)\in M$, all of whose tangent vectors belongs to a tangent angle space distribution,
i.e. $\dot{\g}(t)\in T_\th M$ for every $t\in (a,b)$,
is called an \textit{angle line} of $M$. More generally, an \textit{angle surface} is a connected
submanifold $S$ of $M$, such that for any $p\in S$, $T_p S\subset T_\th M$.

Now we can formulate  the definition of \textit{submanifolds with constant Jordan angles (CJA)}, the main subject of this paper.

\begin{defi}\label{def-CJA}
Let $M$ be an $n$-dimensional submanifold of $\R^{n+m}$ and $Q_0$ be a fixed $m$-plane. If every normal Jordan angle function of $M$
relative to $Q_0$ is a constant function, and $m_\th^N$ is constant on $M$ for each $\th\in \Arg^N$, then
we say $M$ has \textbf{constant Jordan angles (CJA)} relative to $Q_0$.
\end{defi}

With the aid of Lemma \ref{Jordan} and Proposition \ref{position}, one can obtain equivalent definitions of submanifolds with CJA.

\begin{pro}
For any $n$-dimensional submanifold $M$ of $\R^{n+m}$ and a fixed $m$-plane $Q_0$, the following statement are equivalent:

(i) $M$ has CJA relative to $Q_0$;

(ii) Every tangent angle function of $M$ relative to $Q_0^\bot$ is constant, and $m_\th^T$ is constant;

(iii) $\Arg(N_p M,Q_0)$ ($\Arg(T_p M,Q_0^\bot)$) is independent of $p\in M$, and the multiplicity of each normal (tangent) Jordan angle
is constant;

(iv) The relative position of $N_p M$ ($T_p M$) and $Q_0$ ($Q_0^\bot$) is independent of $p\in M$.
\end{pro}

\noindent {\bf Remarks:}
\begin{itemize}
\item  Let $\g$ be an arc-length parameterized curve in $\R^3$. If $\g$ is a \textit{constant angle curve},
 i.e. the unit tangent vector at every point makes a constant angle with a fixed straight line in
$\R^3$, then $\g$ is a helix, and vise versa.  Let $S$ be a smooth surface in $\R^3$, if the normal vector at every point makes a constant angle with a fixed straight line in
$\R^3$, then $S$ is said to be a \textit{constant angle surface} in $\R^3$. A surface $S$ in $\R^3$ is a constant angle surface if and only if it is locally isometric to either a cylinder, a right circular cone, or the tangential developable of a helix. Moreover, if we additionally assume $S$ to be complete, then $S$ has to be a cylinder. Recently, many geometers are interested in constant angle surfaces in other ambient spaces, e.g.
$S^2\times \R$ \cite{d-f-v-v}, $\Bbb{H}^2\times \R$ \cite{d-m}, Heisenberg group
\cite{f-m-v}, Minkowski space \cite{l-m} and product spaces \cite{d-k}. Our notion is a natural generalization of
the classical constant angle curves and surfaces.

\item If $M^n$ is a hypersurface of $\R^{n+1}$, then $M$ has CJA if and only if $M$ is a helix hypersurface \cite{d-r}. Hence the concept of submanifolds
with CJA is a natural generalization of helix hypersurfaces to higher codimensional cases. Helix hypersurfaces are  closed
related to the shadow problem (see \cite{g}) formulated by H. Wente, and another interesting motivation for the study of helix
hypersurfaces comes from the physics of interfaces of liquid crystal  (see \cite{c-d}).

\item Let $S$ be a surface in $\R^4$, then $S$ has CJA if and only if $S$ is a surface in $\R^4$ with \textit{constant principal angles with respect to
a plane}. This concept was introduced by Bayard-Di Scala-Castro-Hern\'{a}ndez in \cite{b-d-c-h}. In this paper, the authors established
a local existence theorem and classified all the complete surfaces in $\R^4$ with constant principal angles.
\end{itemize}

 Denote
\begin{equation}
r:=\sum_{\th \in \Arg^N,\th\neq 0} m_\th^N,
\end{equation}
then $r$ is a constant $\Bbb{Z}^+$-valued function on $M$. As shown above, $0\in \Arg^N$ ($0\in \Arg^T$)
if and only if $r<m$ ($r<n$), and the multiplicity of $0$ equals $m-r$ ($n-r$). Let
\begin{equation}
g^N:=|\text{Arg}^N|\qquad g^T:=|\text{Arg}^T|
\end{equation}
be the numbers of distinct Jordan angles. Note that $g^N=g^T+1$ whenever $r\equiv n<m$,
$g^T=g^N+1$ whenever $r\equiv m<n$, and otherwise $g^N=g^T$.

In conjunction with Lemma \ref{Jordan} and Lemma \ref{dis},
$NM$ and $TM$ have the following vector bundle decompositions
\begin{equation}\aligned
NM&=\bigoplus_{\th\in \text{Arg}^N}N_{\th}M,\\
TM&=\bigoplus_{\th\in \text{Arg}^T}T_{\th}M.
\endaligned
\end{equation}
In particular, if $\th\neq 0, \pi/2$, then there exists a smooth mapping $\Phi_\th: R_\th M\ra R_\th M$, where
\begin{equation}
R_\th M:=N_\th M\oplus T_\th M,
\end{equation}
such that: (i) $\Phi_\th$ keeps each fiber invariant; (ii) the length of each vector in $R_\th M$ is invariant
under $\Phi_\th$; (iii) $\Phi_\th^2=-\mb{Id}$; (iv) $\Phi_\th(N_\th M)=T_\th M$, $\Phi_\th(T_\th M)=N_\th M$; (iv) for any $\nu\in N_\th M$ and $u\in T_\th M$,
\begin{equation}\label{Phi2}\aligned
\sec\th\ \mc{P}_0\nu&=\cos\th\ \nu-\sin\th\ \Phi_\th(\nu),\\
\sec\th\ \mc{P}_0^\bot u&=\cos\th\ u-\sin\th\ \Phi_\th(u).
\endaligned
\end{equation}
$\Phi_\th$ is called the \textit{anti-involution} associated to $\th$.

\subsection{Minimal submanifolds with CJA and the Bernstein problem}
The concept of CJA submanifolds that we have just introduced arises from our systematic investigation of
 the Bernstein problem in higher codimension. We now wish to explain this connection.

The classical Bernstein theorem \cite{be} states that any entire minimal graph in $\R^3$ has to be affine linear.
This result has been extended by J. Simons \cite{Si} to such entire minimal graphs in $\R^{n+1}$ for $n\leq 7$, whereas
Bombieri-de Giorgi-Giusti \cite{b-g-g} constructed counterexamples in higher dimensions. But for any dimensions, there is a weak
version of the Bernstein type theorem, obtained by  J. Moser \cite{m} who proved that any entire solution $f:\R^n\ra \R$ to the minimal surface equation
\begin{equation}\label{min-eq}
\text{div}\Big(\f{\n f}{\sqrt{1+|\n f|^2}}\Big)=0
\end{equation}
has to be affine linear, provided that
\begin{equation}v:=\sqrt{1+|\n f|^2}
\end{equation}
is a bounded function.  $v$ is a significant quantity here for various reasons. Firstly,
the boundedness of $v$ ensures that (\ref{min-eq}) is a uniformly elliptic equation, so that a Bernstein type result can be obtained by  Moser's iteration. Secondly, for any $f:\R^n\ra \R$,
$x=(x^1,\cdots,x^n)\in \R^n\mapsto (x,f(x))\in \text{graph } f$ is a global coordinate chart of the graph
of $f$, and a straightforward calculation shows that the volume form of $\text{graph }f$ is $vdx^1\w \cdots\w dx^n$,
i.e.  $v$ equals the radio of the volume form of  $\text{graph }f$ and the coordinate plane. Thirdly, $v$ has a close
relationship with Jordan angles. A direct computation shows
\begin{equation}
\nu:=w(-\pd{f}{x^1},\cdots,-\pd{f}{x^n},1)\qquad \text{where }w:=v^{-1}
\end{equation}
is a unit normal vector field on $\text{graph }f$. Thus the angle between $\nu$ and the $x^{n+1}$-axis is $\arccos w$,
which is smaller than an acute angle whenever the $v$-function is bounded. Therefore, Moser's theorem can be
restated as: Let $M$ be a complete minimal hypersurface in $\R^{n+1}$ and $\th_0\in (0,\pi/2)$. If the angle between the normal vector
and $x^{n+1}$-axis is smaller than $\th_0$ everywhere, then $M$ has to be an affine $n$-plane.

Now we consider an $n$-dimensional entire minimal graph $M$ in $\R^{n+m}$, generated by a smooth vector-valued function
$f:\R^n\ra \R^m$
$$x=(x^1,\cdots,x^n)\mapsto f(x)=(f^1(x),\cdots,f^m(x)).$$
Then $f$ satisfies the minimal surface equations
\begin{equation}
\aligned
\sum_i \pf{x^i}(v g^{ij})=0\qquad&\forall j=1,\cdots,n,\\
\sum_{i,j}\pf{x^i}(v g^{ij}\pd{f^\a}{x^j})=0\qquad &\forall \a=1,\cdots,m.
\endaligned
\end{equation}
Here $g_{ij}dx^i dx^j$ is the induced metric on $M$, $(g^{ij})$ denotes the inverse matrix of $(g_{ij})$,
and $vdx^1\w \cdots\w dx^n:=\det(g_{ij})^{1/2}dx^1\w\cdots\w dx^n$ is the volume form of $M$. More precisely,
\begin{equation}
v=\Big[\det\big(\de_{ij}+\sum_\a \pd{f^\a}{x^i} \pd{f^\a}{x^j}\big)\Big]^{1/2}.
\end{equation}
Similarly to the case of codimension 1, the $v$-function has a close relationship with Jordan angles. At any point $p\in M$,
denote by
\begin{equation}
0\leq \th_1\leq\th_2\leq \cdots\leq \th_m<\pi/2
\end{equation}
the Jordan angles between $N_p M$ and the coordinate $m$-plane,
then a calculation shows (see \cite{x-y1}\cite{j-x-y2})
\begin{equation}\label{W}
v=\prod_{i=1}^m \sec\th_m.
\end{equation}
We note that
\begin{equation}\label{w}
w:=v^{-1}=\prod_{i=1}^m \cos\th_m
\end{equation}
is the inner product of the normal $m$-plane and the coordinate $m$-plane. Here all the $m$-planes are viewed as vectors in
a Euclidean space of larger dimension, via Pl\"{u}cker embedding (see \cite{j-x-y3}).

It is natural to ask whether Moser's theorem can be generalized to the higher codimensional case. In other words,
given an entire minimal graph $M=\text{graph }f\subset \R^{n+m}$ with $f:\R^n\ra \R^m$, does the boundedness of the
$v$-function ensure that $M$ has to be an affine $n$-plane? The answer is 'Yes' for the cases of dimension 2 \cite{c-o}\cite{j-x-y5}
and dimension 3 \cite{b}\cite{fc}, but it is 'No' for dimension 4, according to the works
of Lawson-Osserman \cite{l-o} and Harvey-Lawson \cite{h-l}.

Let us explain the geometric reason for this fact.
Let $\Bbb{O}$ and $\Bbb{H}$ denote the octonions and quaternions, respectively. We have $\Bbb{O}=\Bbb{H}\oplus \Bbb{H}e$, with $e$ a unit element orthogonal to $\Bbb{H}$,
and for any $a,b,c,d\in \H$,
\begin{equation}
 (a+be)(c+de)=(ac-\bar{d}b)+(da+b\bar{c})e.
\end{equation}
Denote $\text{Sp}_1:=\{q\in \Bbb{H}:|q|=1\}$. Assume $a\in \Im\H$ is a fixed unit element, then
\begin{equation}\label{l-o-cone}
M(a):=\{r\big[(\sqrt{5}/2) qa \bar{q}+\bar{q}e\big]:q\in \text{Sp}_1,r\in \Bbb{R}^+\}
\end{equation}
is a $4$-dimensional cone in $\Im\O$, which is the graph of the function $\eta:\H\backslash \{0\}\ra \Im\H\backslash\{0\}$
\begin{equation}
\eta(x)=\f{\sqrt{5}}{2|x|}\bar{x}\ep x.
\end{equation}
Here $\ep\in \Im\H$ and $|\ep|=1$.
Note that $\eta$ is a \textit{cone-like function}, i.e. $\eta(tx)=t\eta (x)$ for any $t$ and $x$.
It was discovered by Lawson-Osserman \cite{l-o} that $\eta$ is a Lipschitz solution to the non-parametric minimal
surface equations that is not $C^1$, and a straightforward calculation shows the $v$-function is always $9$ on
$M(a)$. Afterwards, Harvey-Lawson \cite{h-l} constructed a family of $4$-dimensional entire minimal graphs in $\Im\O$; the tangent cone
at infinity of each one is just the Lawson-Osserman's cone, and the $v$-function takes value in $[1,9)$. Therefore, Moser's theorem cannot
generalize to all higher codimensional cases.

Now we further explore the geometric properties of Lawson-Osserman's cone via Jordan angles.

\begin{pro}\label{lo-cone}
Lawson-Osserman's cone $M(a)$ is a $4$-dimensional submanifold in $\Im \O$ with CJA relative to $\Im \H$, and $\Arg^N=\{\arccos(2/3),\arccos(\sqrt{6}/6)\}$,
$\Arg^T=\{\arccos(2/3),\arccos(\sqrt{6}/6),0\}$.
\end{pro}

\noindent \textbf{Remark.} Proposition \ref{lo-cone} was firstly proved in the Appendix of \cite{j-x-y3}, and the calculation was based on the complex form
of the Hopf map from $S^3$ to $S^2$. Now, we shall give another proof, which is based on the fact that $M(a)$ is a
$\text{Sp}_1$-invariant manifold and has a close relationship with the argument in Section \ref{co-cja}.

\begin{proof}
Denote $F:\text{Sp}_1\times \R^+\ra M(a)$
\begin{equation}
 (q,r)\mapsto r\big[(\sqrt{5}/2) qa \bar{q}+\bar{q}e\big].
\end{equation}
Let $p_0=F(q_0,R_0)$ be an arbitrary point in $M(a)$. We shall compute the Jordan angles between $T_{p_0}M(a)$ and $\H e$.

Let $\text{sp}_1$ be the Lie algebra associated to $\text{Sp}_1$, which can be seen as the linear space constisting of right-invariant
vector fields on $\text{Sp}_1$. It is well-known that $\text{sp}_1$ is isomorphic to $\Im\H$, and the isomorphism is given by
$\chi: \Im \H\ra \text{sp}_1$
\begin{equation}
 b\mapsto V=\chi(b)\text{ with }V_q=\f{d}{dt}\Big|_{t=0}e^{tb}q.
\end{equation}
As a matter of convenience, $b$ and $\chi(b)$ are regarded to be same in the sequel. Then at $p_0$,
$$F_* \p_r=(\sqrt{5}/2)q_0 a\bar{q}_0+\bar{q}_0 e=(\sqrt{5}/2)a_1+\ep$$
with
$$a_1:=q_0 a\bar{q}_0,\quad \ep:=\bar{q}_0 e$$
and
$$\aligned
F_* b&=\f{d}{dt}\Big|_{t=0} R_0\big[(\sqrt{5}/2)(e^{tb}q_0)a(\ol{e^{tb}q_0})+(\ol{e^{tb}q_0})e)\big]\\
&=\f{d}{dt}\Big|_{t=0}R_0\big[(\sqrt{5}/2)e^{tb}(q_0 a\bar{q}_0)e^{-tb}+(\bar{q}_0 e^{-tb})e\big]\\
&=\f{d}{dt}\Big|_{t=0}R_0\big[(\sqrt{5}/2)e^{tb}a_1e^{-tb}+e^{-tb}\ep\big]\\
&=R_0\big[(\sqrt{5}/2)(ba_1-a_1 b)-b\ep\big].
\endaligned$$
Let $a_2$ be a unit vector in $\Im \H$ that is orthogonal to $a_1$
and denote $a_3:=a_1a_2$. Then $\{a_1,a_2,a_3\}$ is an orthonormal basis of $\Im \H$, satisfying $a_1^2=a_2^2=a_3^2=-1$, $a_1 a_2=a_3=-a_2 a_1$,
$a_2a_3=a_1=-a_3a_2$ and $a_3a_1=a_2=-a_1a_3$, then
$$\aligned
R_0^{-1}F_*a_1&=(\sqrt{5}/2)(a_1^2-a_1^2)-a_1\ep=-a_1\ep,\\
R_0^{-1}F_*a_2&=(\sqrt{5}/2)(a_2 a_1-a_1 a_2)-a_2\ep=-\sqrt{5}a_3-a_2\ep,\\
R_0^{-1}F_*a_3&=(\sqrt{5}/2)(a_3 a_1-a_1 a_3)-a_3\ep=\sqrt{5}a_2-a_3\ep.
\endaligned$$
Denote
\begin{equation}
 \aligned
e_1:=&(2/3)F_* \p_r=(\sqrt{5}/3)a_1+(2/3)\ep,\\
e_2:=&(\sqrt{6}R_0)^{-1}F_* a_2=-(\sqrt{30}/6)a_3-(\sqrt{6}/6)a_2\ep,\\
e_3:=&(\sqrt{6}R_0)^{-1}F_* a_3=(\sqrt{30}/6)a_2-(\sqrt{6}/6)a_3\ep,\\
e_4:=&R_0^{-1}F_* a_1=-a_1\ep.
\endaligned
\end{equation}
Then $\{e_1,e_2,e_3,e_4\}$ is an orthonormal basis of $T_{p_0}M(a)$.

Let $\mc{P}_0$, $\mc{P}_0^\bot$ be
the orthogonal projections of $\Im \O=\Im \H\oplus \H e$ into $\Im \H$ and $\H e$,
respectively, then
$$\lan (\mc{P}_0^\bot e_1)^T,e_j\ran=\lan \mc{P}_0^\bot e_1,e_j\ran=\lan \mc{P}_0^\bot e_1,\mc{P}_0^\bot e_j\ran=(4/9)\de_{1j}$$
which implies $(\mc{P}_0^\bot e_1)^T=(4/9)e_1$ and hence $e_1$ is a tangent angle direction associated to
$\th_1:=\arccos (2/3)$. Note that $e_1$ is the direction of the ray going through $p_0$.
Similarly, one can prove that $e_2, e_3$ are both tangent angle directions associated to
$\th:=\arccos (\sqrt{6}/6)$, and $e_4$ is a tangent angle direction associated to $0$.
Since $p_0$ can be taken arbitrarily, $M(a)$ has CJA relative to $\Im \H$, and
$\text{Arg}^T=\{\th_1,\th,0\}$, $\text{Arg}^N=\{\th_1,\th\}$. Moreover,
an arbitrary angle line with respect to $\th_1$ is a ray of $M(a)$, and vise versa.
\end{proof}

In \cite{l-o}, Lawson-Osserman raised the following question: What
is the largest constant $C$ such that an entire minimal graph of
arbitrary dimension and codimension with  $v\leq C$ has to  be
affine linear? Up to now, the best positive answer to this question
in a successive series of achievements by several mathematicians
(see \cite{h-j-w}, \cite{j-x}, \cite{x-y1}, \cite{j-x-y2}) is gotten
in \cite{j-x-y4}, which says that for any entire minimal graph
$M=\{(x,f(x)):x\in \R^n\}\subset \R^{n+m}$ with $f:\R^n\ra \R^m$, if
$v\leq 3$, then $M$ has to be an affine $n$-plane. But there is
still a large quantitative gap between 3 and 9, that is, between
known Bernstein type theorems and the counterexamples.

Lawson-Osserman's problem can be viewed as the first gap problem of the $v$-function for entire minimal graphs of higher codimension.
To study the gap phenomena of the $v$-function, it is natural to consider minimal graphs whose $v$-function is constant.
Observing that the $v$-function is a function of all Jordan angle functions (see (\ref{W})), the $v$-function on any minimal
graph with CJA relative to the coordinate plane is constant.
Proposition \ref{lo-cone} shows the Lawson-Osserman's cone $M(a)$ has CJA relative to the imaginary quaternions, but unfortunately it is not a complete submanifold.
So one can propose the following problems:

\begin{problem}
Do there exist nonflat entire minimal graphs whose $v$-function is constant?
\end{problem}

\begin{problem}\label{problem1}
Let $S_v$ and $S_v^0$ be sets consisting of some real numbers strictly bigger than 1. $v_0\in S_v$ if and only if there exists a nonlinear cone-like
map $f:\R^n\backslash \{0\}\ra \R^m\backslash \{0\}$, such that $M=\text{graph } f$ is a minimal graph whose $v$-function always equals
$v_0$. Similarly, $v_0 \in S_v^0$ if and only if there exists a nonlinear cone-like
map $f:\R^n\backslash \{0\}\ra \R^m\backslash \{0\}$, such that $M=\text{graph } f$ is a minimal graph with CJA relative to $\R^m$. Are $S_v$ and $S_v^0$ discrete sets?
\end{problem}

\begin{problem}\label{problem2}
Let $S_{v,loc}$ and $S_{v,loc}^0$ be sets consisting of some real numbers strictly bigger than 1. $v_0\in S_{v,loc}$ if and only if there exists a nonlinear
vector-valued function $f:D\subset \R^n \ra \R^m$ ($D$ is an open domain), such that $M=\text{graph } f$ is a minimal graph whose $v$-function always equals
$v_0$. Similarly, $v_0\in S_{v,loc}^0$ if and only if there exists a nonlinear
vector-valued function $f:D\subset \R^n \ra \R^m$, such that $M=\text{graph }f$ is a minimal graph with CJA relative to $\R^m$. Are $S_{v,loc}$ and $S_{v,loc}^0$ discrete sets?
\end{problem}

\begin{problem}
Does any minimal graph in Euclidean space with constant $v$-function have to be a submanifold with CJA?
\end{problem}

Obviously $S_v^0\subset S_v$, $S_{v,loc}^0\subset S_{v,loc}$, $S_v\subset S_{v,loc}$, $S_v^0\subset S_{v,loc}^0$ and
Problem \ref{problem2} can be viewed as a local version of Problem \ref{problem1}. For Problem \ref{problem1}, the known facts include $(1,3]\notin S_v$ (see \cite{j-x-y4}) and $9\in S_v^0$.

Problem \ref{problem1} is quite similar to Chern's conjecture, intrinsic rigidity problem in the theory of minimal submanifolds, which claims
that if the squared length of the second fundamental form (denoted by $|B|^2$) of a compact minimal submanifold in the unit Euclidean sphere is constant, then
the value should be contained in a discrete set (see \cite{c-c-k}).

\subsection{Submanifolds in spheres with CJA}
If $M$ is an $n$-dimensional cone in $\R^{n+m}$, then the intersection of $M$ and the unit sphere gives
an $(n-1)$-dimensional submanifold $N$ in $S^{n+m-1}$. $M$ is said to be the cone generated by $N$, i.e.
$M=CN$. As pointed out by J. Simons \cite{Si}, the geometric properties of $N$ are closed related to those of the cone
$CN$. Firstly, $CN$ has parallel mean curvature in $\R^{n+m}$ if and only if $N$ is a minimal submanifold in $S^{n+m-1}$
(see \cite{x0} p.64). Noting that $CN$ is a linear subspace if and only if $N$ is a totally geodesic subsphere,
the Bernstein problem for minimal submanifolds in Euclidean space can be transferred to the spherical Bernstein problem
for minimal submanifolds in the sphere, in the framework of the geometric measure theory (see \cite{a}, \cite{f}).

For any $p\in N$, denote by $T_p N$ and $N_p N$ the tangent $(n-1)$-plane and the normal $m$-plane of $N$ at $p$, respectively,
then
$$T_p S^{n+m-1}=T_p N\oplus N_p N.$$
Along the ray going through $p$, the tangent $n$-planes and the normal $m$-planes of $CN$ are both constant, and
\begin{equation}\label{sphere}
T_p (CN)=T_p N\oplus \{t\mathbf{X}(p):t\in \R\},\quad N_p (CN)=N_p N.
\end{equation}
Here $\mathbf{X}(p)$ denotes the position vector of $p$ in $\R^{n+m}$.

Let $Q_0$ be a fixed $m$-plane in $\R^{n+m}$, if $\Arg(N_p N,Q_0)$
is independent of $p\in N$, and the multiplicity of each normal Jordan angle is constant, then we say
$N$ is a \textbf{submanifold in a sphere with constant Jordan angles (CJA)} relative to $Q_0$. By
(\ref{sphere}), $N$ has CJA if and only if the cone $CN$ generated by $N$ is a submanifold in $\R^{n+m}$
with CJA.

Thereby, Problem \ref{problem2}  can be restated as follows:

\begin{problem}
Let $S_w$ and $S_w^0$ be sets consisting of some real numbers taking values in $(0,1)$. $w_0\in S_w$
if and only if there exists an $(n-1)$-dimensional compact minimal submanifold $N$ in $S^{n+m-1}$,
that is non-totally geodesic, such that its $w$-function always equals $w_0$, i.e. the inner product
of each normal $m$-plane and a fixed $m$-plane $Q_0$ is $w_0$. Similarly, $w_0\in S_w^0$
if and only if there exists an $(n-1)$-dimensional compact minimal and non-totally geodesic submanifold $N$ in $S^{n+m-1}$,
which has CJA relative to a fixed $m$-plane, such that its $w$-function always equals $w_0$. Are $S_w$ and $S_w^0$ discrete?
\end{problem}

\noindent \textbf{Remark.}  Due to (\ref{w}),
$$S_w=\{w_0=v_0^{-1}:v_0\in S_v\},\quad S_w^0=\{w_0=v_0^{-1}:v_0\in S_v^0\}.$$

There is a long way  to resolving these problems. In this paper, we only consider CJA submanifolds with a small number of distinct Jordan angles (i.e. $g^N$ and $g^T$).

\subsection{Main results}
This paper will be organized as follows.

In Section \ref{s4}, the second fundamental form $B$ of submanifolds with CJA in Euclidean space shall be studied. At first, differentiating the Jordan angle
functions not only gives some nullity properties of $B$, but also reveals the relationship between the induced tangent (normal) connection
and the second fundamental form. Taking the covariant derivative of the formulas obtained in the previous step, one can compute
some components of $\n B$ in terms of $B$. With the aid of the Codazzi equations, we can derive
a constraint equation for the second fundamental form (see Lemma \ref{sec30}), which is nontrivial when the multiplicity of a tangent Jordan angle function $\th\in (0,\pi/2)$, i.e. $m_\th^T$, is
strictly larger than 1.
This conclusion will play an important part in Section \ref{co-cja}. Based on these formulas, it is easy to get some vanishing theorems for the second fundamental form $B$
of submanfolds with CJA, including the following one.

\begin{thm}\label{cja1}
Let $f$ be an $\R^m$-valued function on an open domain $D\subset \R^n$. If $M=\text{graph }f$ is a minimal submanifold with CJA relative
to $\R^m$, and $g^N,g^T\leq 2$, then $f$ has to be affine linear, i.e. $M$ has to be an affine $n$-plane.

\end{thm}

Note that the example of Lawson-Osserman's cone implies that
the condition '$g^N,g^T\leq 2$' in Theorem \ref{cja1} cannot be omitted.

In \cite{h-l}, Harvey-Lawson introduced a new concept of coassociative submanifolds,
as an important example of calibrated geometries, and showed that Lawson-Osserman's cone
is a coassociative submanifold. Observing that coassociative submanifolds constitute an important class
of 4-dimensional minimal submanifolds in $\R^7$, it is natural to study the structure
of coassociative submanifolds with CJA, which is the main topic of Section \ref{co-cja}.
With the aid of the algebraic properties of octonions, one can obtain several interesting conclusions on
  the Jordan angles and the second fundamental form of
coassociative submanifolds. In conjunction with Lemma \ref{sec30}, a structure theorem
for coassociative submanifolds with CJA is deduced as follows.

\begin{thm}
Let $f$ be a smooth function from an open domain $D\subset \H$ into $\Im \H$. If $M=\text{graph }f$ is
a coassociative submanifold with CJA relative to $\Im \H$, and $g^N\leq 2, g^T\leq 3$,
then $f$ is either an affine linear function or
$f(x)=\eta(x-x_0)+y_0$,
where $x_0\in \H$, $y_0\in \Im \H$ and
$$\eta(x)=\f{\sqrt{5}}{2|x|}\bar{x}\ep x$$
with $\ep$ an arbitrary unit element in $\Im \H$. In other words, $M$ is an affine 4-plane or a translate of an open subset of the Lawson-Osserman's cone.
\end{thm}

\bigskip\bigskip

\Section{On the second fundamental form of submanfolds with CJA}{On the second fundamental form of submanfolds with CJA}\label{s4}

Let $M$ be an $n$-dimensional submanifold in $\R^{n+m}$ with CJA relative to a fixed $m$-plane
 $Q_0$. We use the  notations  $\mc{P}_0$, $\mc{P}_0^\bot$, $\text{Arg}^N$, $\text{Arg}^T$, $N_\th M$, $T_\th M$, $R_\th M$, $m_\th^N$, $m_\th^T$, $g^N$, $g^T$  established
in Section \ref{cja}. For   $p$ in $M$, we put
\begin{equation}
N_{p,\th}M:=N_p M\cap N_\th M,\qquad T_{p,\th}M:=T_p M\cap T_\th M.
\end{equation}

The second fundamental form
$B$ is a pointwise symmetric bilinear form on $T_p M$ ($p\in M$) with values in $N_p M$ defined by
$$B_{XY}=(\overline{\n}_X Y)^N$$
with $\overline{\n}$ the Levi-Civita connection on $\R^{n+m}$. The induced connections on $TM$ and $NM$ are
$$\n_X Y=(\overline{\n}_X Y)^T,\qquad \n_X \nu=(\overline{\n}_X \nu)^N.$$
Here $X,Y$ are smooth sections of $TM$ and $\nu$ denotes a smooth section of $NM$.
The second fundamental form, the curvature tensor of the submanifold, the curvature tensor of the normal bundle and
the curvature tensor of the ambient manifold satisfy the Gauss, Codazzi and Ricci equations
(see \cite{x} for details).

Let $A$ be the shape operator defined by
\begin{equation}
A^\nu(v)=(-\ol{\n}_v\nu)^T\qquad \forall\nu\in \G(NM),v\in T_p M.
\end{equation}
 $A^\nu$ is a symmetric operator on $T_p M$ and  satisfies the Weingarten equations
\begin{equation}
\lan B_{XY},\nu\ran=\lan A^\nu(X),Y\ran\qquad \forall X,Y\in \G(TM).
\end{equation}

The trace of the second fundamental form gives a normal vector field $H$ on $M$, which is called the mean curvature vector field.
If $\n H\equiv 0$, then we say that $M$ has parallel mean curvature. Moreover if $H\equiv 0$, $M$ is called a minimal submanifold.

\subsection{Nullity lemmas}
Let $\th\in \text{Arg}^N$($\th\neq 0,\pi/2$) and $\Phi_\th: R_\th M\ra R_\th M$ denote the anti-involution associated to $\th$,
then (\ref{Phi2}) gives
\begin{equation*}
\aligned
\mc{P}_0^\bot v&=\cos\th\big(\cos\th\ v-\sin\th\ \Phi_{\th}(v)\big)\\
&=\cos^2\th\ v-\cos\th\sin\th\ \Phi_{\th}(v)
\endaligned
\end{equation*}
for any $v\in T_{\th}M$ and
$$\aligned
\mc{P}_0 \mu&=\cos\th\big(\cos\th\ \mu-\sin\th\ \Phi_\th(\mu)\big)\\
&=\cos^2\th\ \mu-\cos\th\sin\th\ \Phi_\th(\mu)
\endaligned$$
for any $\mu\in N_\th M$.
 In other words,
\begin{equation}\label{ang1}
\aligned
(\mc{P}_0^\bot v)^T=\cos^2\th\ v,&\qquad (\mc{P}_0^\bot v)^N=-\cos\th\sin\th\ \Phi_\th(v),\\
(\mc{P}_0\mu)^N=\cos^2\th\ \mu,&\qquad (\mc{P}_0 \mu)^T=-\cos\th\sin\th\ \Phi_\th(\mu).
\endaligned
\end{equation}
Based on the above formulas, one can easily deduce the following nullity lemmas for the second fundamental form
of $M$.
\bigskip

\begin{lem}\label{sec2}
For each $\th\in \text{Arg}^N$ which takes values in  $(0,\pi/2)$,
\begin{equation}
\lan B_{uv},\Phi_\th(w)\ran+\lan B_{uw},\Phi_\th(v)\ran=0
\end{equation}
holds pointwisely for any $u\in T_p M$ and $v,w\in T_{p,\th}M$. In particular,
\begin{equation}\label{sec1}
\lan B_{uv},\Phi_\th(v)\ran=0
\end{equation}
for every $v\in T_{p,\th}M$.
\end{lem}

\begin{proof}
By linearity, it  suffices to prove (\ref{sec1})
for any unit vector $v\in T_{p,\th}M$.

Let $X$ be a smooth local section of $T_\th M$, such that $X_p=v$ and $|X|\equiv 1$, then
\begin{equation}\label{ang2}
\lan \mc{P}_0^\bot X,\mc{P}_0^\bot X\ran=|\mc{P}_0^\bot X|^2\equiv \cos^2\th.
\end{equation}
Differentiating both sides with respect to $u$ yields
$$\aligned
0&=(1/2)\n_u \lan \mc{P}_0^\bot X,\mc{P}_0^\bot X\ran=\lan \overline{\n}_u (\mc{P}_0^\bot X),\mc{P}_0^\bot v\ran\\
&=\lan \mc{P}_0^\bot(\ol{\n}_u X),\mc{P}_0^\bot v\ran=\lan \mc{P}_0^\bot(\n_u X),\mc{P}_0^\bot v\ran+\lan \mc{P}_0^\bot B_{uv},\mc{P}_0^\bot v\ran\\
&=\lan \n_u X,(\mc{P}_0^\bot v)^T\ran+\lan B_{uv},(\mc{P}_0^\bot v)^N\ran\\
&=\cos^2\th \lan \n_u X,v\ran-\cos\th\sin\th \lan B_{uv},\Phi_\th(v)\ran\\
&=(1/2)\cos^2\th \n_u |X|^2-\cos\th\sin\th \lan B_{uv},\Phi_\th(v)\ran\\
&=-\cos\th\sin\th\lan B_{uv},\Phi_\th(v)\ran
\endaligned$$
(where we have used (\ref{ang1})) and then we arrive at (\ref{sec1}).

\end{proof}
\bigskip

\begin{lem}\label{sec6}
For each $\th\in \text{Arg}^N$ taking values in $(0,\pi/2)$,
\begin{equation}\label{sec3}
\lan B_{uv},\nu\ran=0
\end{equation}
for any $u,v\in T_{p,\th}M$ and $\nu\in N_{p,\th}M$.
\end{lem}

\begin{proof}
Let $w:=-\Phi_\th(\nu)$, then $w\in T_{p,\th}M$ and $\Phi_\th(w)=-\Phi_\th^2(\nu)=\nu$. Applying Lemma
\ref{sec2} gives
$$\aligned
\lan B_{uv},\nu\ran&=\lan B_{uv},\Phi_\th(w)\ran=-\lan B_{uw},\Phi_\th(v)\ran\\
&=-\lan B_{wu},\Phi_\th(v)\ran=\lan B_{wv},\Phi_\th(u)\ran\\
&=\lan B_{vw},\Phi_\th(u)\ran=-\lan B_{vu},\Phi_\th(w)\ran\\
&=-\lan B_{uv},\Phi_\th(w)\ran=-\lan B_{uv},\nu\ran
\endaligned$$
and (\ref{sec3}) immediately follows from the above equation.

\end{proof}

\bigskip

\begin{lem}\label{sec25}
If $\th\in \text{Arg}^N\cap \text{Arg}^T$ and $\th\equiv 0$ or $\pi/2$, then
\begin{equation}
\lan B_{uv},\nu\ran=0
\end{equation}
for any $u\in T_p M$, $v\in T_{p,\th}M$ and $\nu\in N_{p,\th}M$.
\end{lem}

\begin{proof}

If $\th\equiv 0$, let $X$ be a smooth local section of $T_\th M$ such that $X_p=v$, then $X_q\in Q_0^\bot$
for any $q$. Thus $(\ol{\n}_u X)_p\subset Q_0^\bot$. On the other hand, $\nu\in N_{p,\th}M$ implies
$\nu\in Q_0$, hence
$$\lan B_{uv},\nu\ran=\lan \ol{\n}_u X,\nu\ran=0.$$
The proof for $\th\equiv\pi/2$ is similar.
\end{proof}

\subsection{Connections}\label{2.2}

Let $\th,\si\in \text{Arg}^T$, $\th\neq \si$, $X$  a local section of $TM$, $Y$ and $Z$ local sections of
$T_\th M$ and $T_\si M$, respectively. Define
\begin{equation}
(S_{\th\si})_{YZ}(X):=\lan \n_X Y,Z\ran.
\end{equation}
Then for any smooth function $f$ defined on $M$, $(S_{\th\si})_{YZ}(fX)=f(S_{\th\si})_{YZ}(X)$,
$(S_{\th\si})_{Y,fZ}(X)=f(S_{\th\si})_{YZ}(X)$ and
$$\aligned
(S_{\th\si})_{fY,Z}(X)&=\lan \n_X(fY),Z\ran=f\lan \n_X Y,Z\ran+(\n_X f)\lan Y,Z\ran\\
&=f(S_{\th\si})_{YZ}(X).
\endaligned$$
This means $S_{\th\si}$ is a smooth tensor field on $M$ of type $(3,0)$. More precisely, $S_{\th\si}$ is a smooth section
of the tensor bundle $T^*M\otimes T_\th^* M\otimes T_\si^* M$. Since $\n$ is a Levi-Civita connection on $M$,
\begin{equation}\label{S}
\aligned
(S_{\th\si})_{YZ}(X)&=\lan \n_X Y,Z\ran=\n_X\lan Y,Z\ran-\lan \n_X Z,Y\ran\\
&=-\lan \n_X Z,Y\ran=-(S_{\si\th})_{ZY}(X).
\endaligned
\end{equation}

Now we additionally define
\begin{equation}
\Phi_\th|_{R_\th M}=0\qquad \text{whenever }\th\equiv 0\text{ or }\pi/2,
\end{equation}
 then (\ref{ang1}) still holds when $\th=0$ or $\pi/2$.
Let
\begin{equation}\label{k}
\k_{\th\si}:=\f{\sin 2\th}{\cos 2\th-\cos 2\si}
\end{equation}
be a constant depending only on $\th$ and $\si$.
The following result reveals the relationship between $S_{\th\si}$ and the second fundamental form.
\bigskip

\begin{lem}\label{sec5}
Let $\th,\si\in \text{Arg}^T$, $\th\neq \si$, then for any $u\in T_p M$, $v\in T_{p,\th}M$ and $w\in T_{p,\si}M$,
\begin{equation}\label{sec4}
(S_{\th\si})_{vw}(u)=\k_{\si\th}\lan B_{uv},\Phi_\si(w)\ran-\k_{\th\si}\lan B_{uw},\Phi_\th(v)\ran.
\end{equation}

\end{lem}

\begin{proof}
Let $Y,Z$ be smooth local sections of $T_\th M$ and $T_\si M$, respectively, such that $Y(p)=v$, $Z(p)=w$, then $(\mc{P}_0^\bot Y)^T=\cos^2\th\ Y$,
$(\mc{P}_0^\bot Z)^T=\cos^2\si\ Z$. Hence
\begin{equation*}
\aligned
0&=\cos^2\th\lan Y,Z\ran=\lan (\mc{P}_0^\bot Y)^T,Z\ran\\
&=\lan \mc{P}_0^\bot Y,Z\ran=\lan \mc{P}_0^\bot Y,\mc{P}_0^\bot Z\ran.
\endaligned
\end{equation*}
Differentiating both sides of the above equation with respect to $u\in T_p M$ yields
$$\aligned
0&=\n_u\lan \mc{P}_0^\bot Y,\mc{P}_0^\bot Z\ran=\lan \ol{\n}_u(\mc{P}_0^\bot Y),\mc{P}_0^\bot w\ran+\lan \mc{P}_0^\bot v,\ol{\n}_u(\mc{P}_0^\bot Z)\ran\\
&=\lan \mc{P}_0^\bot(\ol{\n}_u Y),\mc{P}_0^\bot w\ran+\lan \mc{P}_0^\bot v,\mc{P}_0^\bot(\ol{\n}_u Z)\ran\\
&=\lan \mc{P}_0^\bot(\n_u Y),\mc{P}_0^\bot w\ran+\lan \mc{P}_0^\bot v,\mc{P}_0^\bot(\n_u Z)\ran+\lan \mc{P}_0^\bot B_{uv},\mc{P}_0^\bot w\ran+\lan \mc{P}_0^\bot v, \mc{P}_0^\bot B_{uw}\ran\\
&=\lan \n_u Y,(\mc{P}_0^\bot w)^T\ran+\lan \n_u Z,(\mc{P}_0^\bot v)^T\ran+\lan B_{uv},(\mc{P}_0^\bot w)^N\ran+\lan B_{uw},(\mc{P}_0^\bot v)^N\ran\\
&=\cos^2\si \lan \n_u Y,w\ran+\cos^2\th\lan \n_u Z,v\ran-\cos\si\sin\si\lan B_{uv},\Phi_\si(w)\ran-\cos\th\sin\th\lan B_{uw},\Phi_\th(v)\ran\\
&=(\cos^2\si-\cos^2\th)(S_{\th\si})_{vw}(u)-\cos\si\sin\si\lan B_{uv},\Phi_\si(w)\ran-\cos\th\sin\th\lan B_{uw},\Phi_\th(v)\ran\\
&=(1/2)(\cos 2\si-\cos 2\th)(S_{\th\si})_{vw}(u)-(1/2)\sin 2\si\lan B_{uv},\Phi_\si(w)\ran-(1/2)\sin 2\th\lan B_{uw},\Phi_\th(v)\ran
\endaligned$$
(we have used (\ref{ang1}) and (\ref{S})), which is equivalent to (\ref{sec4}).

\end{proof}
\bigskip

Similarly, given $u\in T_p M$, $\mu\in \G(N_\th M)$, $\nu\in \G(N_\si M)$ with $\th,\si\in \text{Arg}^N$ and $\th\neq \si$, one can define
\begin{equation}
(S_{\th\si}^N)_{\mu\nu}(u):=\lan \n_u \mu,\nu\ran.
\end{equation}
Then $S_{\th\si}^N$ is a smooth section of $T^*M\otimes N_\th^* M\otimes N_\si^* M$, and
\begin{equation}
\aligned
(S_{\si\th}^N)_{\nu\mu}(u)&=\lan \n_u\nu,\mu\ran=\n_u\lan \nu,\mu\ran-\lan \nu,\n_u \mu\ran\\
&=-\lan \n_u \mu,\nu\ran=-(S_{\th\si}^N)_{\mu\nu}(u).
\endaligned
\end{equation}

Let $\mu,\nu$ be local section of $N_\th M$ and $N_\si M$ respectively, then
\begin{equation}
\aligned
0&=\cos^2\th\lan \mu,\nu\ran=\lan (\mc{P}_0\mu)^N,\nu\ran\\
&=\lan \mc{P}_0\mu,\nu\ran=\lan \mc{P}_0\mu,\mc{P}_0\nu\ran.
\endaligned
\end{equation}
Differentiating both sides of the above equality with respect to $u\in T_p M$, one can use (\ref{ang1}) to get the following result, as in the proof
of Lemma \ref{sec5}.
\bigskip
\begin{lem}\label{sec7}
Given $\th,\si\in \text{Arg}^N$, $\th\neq \si$,
\begin{equation}
(S_{\th\si}^N)_{\mu\nu}(u)=\k_{\th\si}\lan B_{u,\Phi_\th(\mu)},\nu\ran-\k_{\si\th}\lan B_{u,\Phi_\si(\nu)},\mu\ran
\end{equation}
for any $u\in T_p M, \mu\in N_{p,\th}M$ and $\nu\in N_{p,\si}M$.

\end{lem}
\bigskip

\subsection{Computation of $\n B$ and related results}
Let $\th\in \text{Arg}^T$, $\si\in \text{Arg}^N$, and $(\cdot)^\si$ be the orthogonal projection of $N_p M$ onto $N_{p,\si}M$. Define
\begin{equation}\label{R}
R_{\th\si}(v_1,v_2,v_3,v_4):=\lan B_{v_1v_3}^\si,B_{v_2v_4}^\si\ran-\lan B_{v_1v_4}^\si,B_{v_2v_3}^\si\ran
\end{equation}
for any $ v_1,v_2,v_3,v_4\in T_{p,\th}M$.
Then $R_{\th\si}$ is a smooth section of the tensor bundle $T_\th^* M\otimes T_\th^* M\otimes T_\th^*M\otimes T_\th^*M$. Obviously
$R_{\th\si}(v_1,v_2,v_3,v_4)=-R_{\th\si}(v_2,v_1,v_3,v_4)=-R_{\th\si}(v_1,v_2,v_4,v_3)=R_{\th\si}(v_3,v_4,v_1,v_2)$, and
\begin{equation}
\aligned
&R_{\th\si}(v_1,v_2,v_3,v_4)+R_{\th\si}(v_2,v_3,v_1,v_4)+R_{\th\si}(v_3,v_1,v_2,v_4)\\
=&\lan B_{v_1v_3}^\si,B_{v_2v_4}^\si\ran-\lan B_{v_1v_4}^\si,B_{v_2v_3}^\si\ran+\lan B_{v_2v_1}^\si,B_{v_3v_4}^\si\ran\\
 &-\lan B_{v_2v_4}^\si,B_{v_3v_1}^\si\ran+\lan B_{v_3v_2}^\si,B_{v_1v_4}^\si\ran-\lan B_{v_3v_4}^\si,B_{v_1v_2}^\si\ran\\
 =&0.
 \endaligned
 \end{equation}
 Hence $R_{\th\si}$ is a curvature type tensor. Note that $R_{\th\si}=0$ whenever $m_\th^T\equiv 1$.

Let $\th,\si\in \text{Arg}^T$, and define
\begin{equation}\label{U}
U_{\th\si}(v_1,v_2,v_3,v_4)=:\big\lan (A^{\Phi_\th(v_3)}v_1)_\si, (A^{\Phi_\th(v_4)}v_2)_\si\big\ran-\big\lan (A^{\Phi_\th(v_4)}v_1)_\si,(A^{\Phi_\th(v_3)}v_2)_\si\big\ran
\end{equation}
for any $v_1,v_2,v_3,v_4\in T_{p,\th}M$. Here $(\cdot)_\si$ denotes the orthogonal projection of $T_p M$ onto $T_{p,\si}M$.
 Due to Lemma \ref{sec2}, $A^{\Phi_\th(v)}w+A^{\Phi_\th(w)}v=0$
for any $v,w\in T_{p,\th}M$, hence
$U_{\th\si}(v_1,v_2,v_3,v_4)=-U_{\th\si}(v_2,v_1,v_3,v_4)=-U_{\th\si}(v_1,v_2,v_4,v_3)=U_{\th\si}(v_3,v_4,v_1,v_2)$
and $U_{\th\si}=0$ whenever $m_\th^T\equiv 1$. Note, however,  that $U_{\th\si}$ does not satisfy a Bianchi type identity.
\bigskip

\begin{lem}\label{sec30}
Given $\th\in \text{Arg}^T$ taking values in $(0,\pi/2)$,
\begin{equation}\label{sec19}
\sum_{\si\in \text{Arg}^N,\si\neq \th}\k_{\th\si}R_{\th\si}(v,w,v,w)=3\sum_{\si\in \text{Arg}^T,\si\neq \th}\k_{\th\si}U_{\th\si}(v,w,v,w)
\end{equation}
for any $v,w\in T_{p,\th}M$,
and moreover
\begin{equation}\label{sec9}\aligned
\lan (\n_v B)_{ww},\Phi_\th(v)\ran=&(1/3)\sum_{\si\in \text{Arg}^N,\si\neq \th}\k_{\th\si}\left(\lan B_{vv}^\si,B_{ww}^\si\ran+2|B_{vw}^\si|^2\right)\\
&-2\sum_{\si\in \text{Arg}^T,\si\neq \th}\k_{\si\th}\lan B_{vw}^\si,\Phi_\si(A^{\Phi_\th(v)}w)_\si\ran.
\endaligned
\end{equation}

\end{lem}

\begin{proof}
Let
\begin{equation}
u_\si:=(A^{\Phi_\th(v)}w)_\si
\end{equation}
for each $\si\in \text{Arg}^T$, then Lemma \ref{sec2} tells us
\begin{equation}\label{sec8}
(A^{\Phi_\th(w)}v)_\si=-(A^{\Phi_\th(v)}w)_\si=-u_\si
\end{equation}
and moreover
\begin{equation}\aligned
U_{\th\si}(v,w,v,w)&=\big\lan (A^{\Phi_\th(v)}v)_\si, (A^{\Phi_\th(w)}w)_\si\big\ran-\big\lan (A^{\Phi_\th(w)}v)_\si,(A^{\Phi_\th(v)}w)_\si\big\ran\\
&=|u_\si|^2.
\endaligned
\end{equation}
In particular, combining the Weingarten equations and Lemma \ref{sec6} gives
$$|u_\th|^2=\lan u_\th,A^{\Phi_\th(v)}w\ran=\lan B_{u_\th w},\Phi_\th(v)\ran=0,$$
i.e. $u_\th=0$.

Let $Y,Z$ be local sections of $T_\th M$ such that $Y_p=v$, $Z_p=w$. By Lemma \ref{sec6}, $\lan B_{ZZ},\Phi_\th(Y)\ran\equiv 0$, hence
\begin{equation}\label{sec10}
\aligned
\lan (\n_v B)_{ww},\Phi_\th(v)\ran&=\n_v \lan B_{ZZ},\Phi_\th(Y)\ran-\lan B_{ww},\n_v \Phi_\th(Y)\ran-2\lan B_{\n_v Z,w},\Phi_\th(v)\ran\\
&=-\lan B_{ww},\n_v \Phi_\th(Y)\ran-2\lan B_{\n_v Z,w},\Phi_\th(v)\ran\\
&:=-I-2II\endaligned
\end{equation}
where
\begin{equation}\aligned\label{sec11}
I&=\sum_{\si\in \text{Arg}^N}\lan B_{ww}^\si,\n_v \Phi_\th(Y)\ran=\sum_{\si\in \text{Arg}^N,\si\neq \th}(S_{\th\si}^N)_{\Phi_\th(v),B_{ww}^\si}(v)\\
&=\sum_{\si\in \text{Arg}^N,\si\neq \th}\left(\k_{\th\si}\lan B_{v,\Phi_\th^2(v)},B_{ww}^\si\ran-\k_{\si\th}\lan B_{v,\Phi_\si(B_{ww}^\si)},\Phi_\th(v)\ran \right)\\
&=-\sum_{\si\in \text{Arg}^N,\si\neq \th}\k_{\th\si}\lan B_{vv}^\si,B_{ww}^\si\ran
\endaligned
\end{equation}
(Lemma \ref{sec6}, Lemma \ref{sec2}, Lemma \ref{sec7} and $\Phi_\th^2=-\mathbf{Id}$ have been used in this calculation) and
\begin{equation}\aligned\label{sec12}
II&=\lan B_{\n_v Z,w},\Phi_\th(v)\ran=\sum_{\si\in \text{Arg}^T} \lan \n_v Z,u_\si\ran=\sum_{\si\in \text{Arg}^T,\si\neq \th}(S_{\th\si})_{w u_\si}(v)\\
&=\sum_{\si\in \text{Arg}^T,\si\neq \th}\left(\k_{\si\th}\lan B_{vw},\Phi_\si(u_\si)\ran-\k_{\th\si}\lan B_{v u_\si},\Phi_\th(w)\ran\right)\\
&=\sum_{\si\in \text{Arg}^T,\si\neq \th}\left(\k_{\si\th}\lan B_{vw},\Phi_\si(u_\si)\ran+\k_{\th\si}|u_\si|^2\right).
\endaligned
\end{equation}
(Here we have used the Weingarten equations, (\ref{sec8}) and Lemma \ref{sec5}.) Substituting (\ref{sec11}) and (\ref{sec12}) into (\ref{sec10}) implies
\begin{equation}\label{sec17}
\lan(\n_v B)_{ww},\Phi_\th(v)\ran=\sum_{\si\in \text{Arg}^N,\si\neq \th} \k_{\th\si}\lan B_{vv}^\si,B_{ww}^\si\ran-2\sum_{\si\in \text{Arg}^T,\si\neq \th}\left(\k_{\si\th}\lan B_{vw},\Phi_\si(u_\si)\ran+\k_{\th\si}|u_\si|^2\right).
\end{equation}

Again applying Lemma \ref{sec6} gives $\lan B_{ZY},\Phi_\th(Y)\ran\equiv 0$, hence
\begin{equation}\label{sec13}\aligned
\lan (\n_w B)_{wv},\Phi_\th(v)\ran=&\n_w \lan B_{ZY},\Phi_\th(Y)\ran-\lan B_{wv},\n_w \Phi_\th(Y)\ran\\
&-\lan B_{\n_w Z,v},\Phi_\th(v)\ran-\lan B_{w,\n_w Y},\Phi_\th(v)\ran\\
=&-\lan B_{wv},\n_w \Phi_\th(Y)\ran-\lan B_{w,\n_w Y},\Phi_\th(v)\ran\\
:=&-I-II
\endaligned
\end{equation}
where
\begin{equation}\label{sec14}\aligned
I=&\sum_{\si\in \text{Arg}^N}\lan B_{wv}^\si,\n_w \Phi_\th(Y)\ran=\sum_{\si\in \text{Arg}^N,\si\neq \th}(S_{\th\si}^N)_{\Phi_\th(v),B_{wv}^\si}(w)\\
=&\sum_{\si\in \text{Arg}^N,\si\neq \th}\left( \k_{\th\si}\lan B_{w,\Phi_\th^2(v)},B_{wv}^\si\ran-\k_{\si\th}\lan B_{w,\Phi_\si(B_{wv}^\si)},\Phi_\th(v)\ran \right)\\
=&\sum_{\si\in \text{Arg}^N,\si\neq \th}\left(-\k_{\th\si}|B_{wv}^\si|^2-\k_{\si\th}\lan u_\si,\Phi_\si(B_{wv}^\si)\ran\right)
\endaligned
\end{equation}
and
\begin{equation}\label{sec15}\aligned
II=&\lan B_{w,\n_w Y},\Phi_\th(v)\ran=\sum_{\si\in \text{Arg}^T}\lan \n_w Y,u_\si\ran=\sum_{\si\in \text{Arg}^T,\si\neq \th}(S_{\th\si})_{v u_\si}(w)\\
=&\sum_{\si\in \text{Arg}^T,\si\neq \th}\left(\k_{\si\th}\lan B_{wv},\Phi_\si(u_\si)\ran-\k_{\th\si}\lan B_{w u_\si},\Phi_\th(v)\ran\right)\\
=&\sum_{\si\in \text{Arg}^T,\si\neq \th}\left(\k_{\si\th}\lan B_{wv},\Phi_\si(u_\si)\ran-\k_{\th\si}|u_\si|^2\right).
\endaligned
\end{equation}
If $\si\neq 0,\pi/2$, then $\Phi_{\si}$ is isometric and $\Phi_\si^2=-\mathbf{Id}$. Hence
$$\lan u_\si,\Phi_\si(B_{wv}^\si)\ran=\lan \Phi_\si(u_\si),\Phi_\si^2(B_{wv}^\si)\ran=-\lan B_{wv}^\si,\Phi_\si(u_\si)\ran.$$
On the other hand, $\Phi_\si=0$ whenever $\si=0$ or $\pi/2$. Therefore
\begin{equation}\label{sec16}
-\sum_{\si\in \text{Arg}^N,\si\neq \th}\k_{\si\th}\lan u_\si,\Phi_\si(B_{wv}^\si)\ran= \sum_{\si\in \text{Arg}^T,\si\neq \th}\k_{\si\th}\lan B_{wv},\Phi_\si(u_\si)\ran.
\end{equation}
Substituting (\ref{sec14})-(\ref{sec16}) into (\ref{sec13}) yields
\begin{equation}\label{sec18}
\lan (\n_w B)_{wv},\Phi_\th(v)\ran=\sum_{\si\in \text{Arg}^N,\si\neq \th}\k_{\th\si}|B_{wv}^\si|^2+\sum_{\si\in \text{Arg}^T,\si\neq \th}\left(\k_{\th\si}|u_\si|^2-2\k_{\si\th}\lan B_{wv},\Phi_\si(u_\si)\ran\right).
\end{equation}

The Codazzi equations imply $(\n_v B)_{ww}=(\n_w B)_{wv}$. Hence by comparing the right hand sides of (\ref{sec17}) and (\ref{sec18}) we arrive at
\begin{equation}\label{sec24}
\sum_{\si\in \text{Arg}^N,\si\neq \th}\k_{\th\si}\left(\lan B_{vv}^\si,B_{ww}^\si\ran-|B_{vw}^\si|^2\right)=3\sum_{\si\in \text{Arg}^T,\si\neq \th}\k_{\th\si}|u_\si|^2
\end{equation}
and then (\ref{sec19}) immediately follows from the definition of $R_{\th\si}$ and $U_{\th\si}$. Finally (\ref{sec9}) is obtained by substituting (\ref{sec24}) into
(\ref{sec18}).

\end{proof}

\bigskip

\begin{lem}\label{sec31}
We consider $\th\in \text{Arg}^T$ taking values in $(0,\pi/2)$ and $\si\in \text{Arg}^T$ such that $\th\neq \si$. If $U_{\th\si}(v_1,v_2,v_1,v_2)=0$ holds
for any $v_1,v_2\in T_{p,\th}M$, then
\begin{equation}\label{sec20}\aligned
\lan (\n_v B)_{ww},\Phi_\th(v)\ran=&-2\sum_{\tau\in \text{Arg}^T,\tau\neq \th}\k_{\tau\th}\lan B_{vw}^\tau,\Phi_\tau(A^{\Phi_\th(v)}w)_\tau\ran\\
&+\sum_{\tau\in \text{Arg}^N,\tau\neq \th}\k_{\th\tau}|B_{vw}^\tau|^2+\sum_{\tau\in \text{Arg}^T,\tau\neq \th}\k_{\th\tau}|(A^{\Phi_\th(v)}w)_\tau|^2
\endaligned
\end{equation}
for any $v\in T_{p,\th}M$ and $w\in T_{p,\si}M$.
\end{lem}

\begin{proof}
In the sequel we make use of the abbreviation $u_\tau:=(A^{\Phi_\th(v)}w)_\tau$ for any $\tau\in \text{Arg}^T$. By the definition of $U_{\th\si}$,
$$\aligned
0&=U_{\th\si}(u_\th,v,u_\th,v)\\
&=\big\lan (A^{\Phi_\th(u_\th)}u_\th)_\si,(A^{\Phi_\th(v)}v)_\si\big\ran-\big\lan (A^{\Phi_\th(v)}u_\th)_\si,(A^{\Phi_\th(u_\th)}v)_\si\big\ran\\
&=|(A^{\Phi_\th(v)}u_\th)_\si|^2
\endaligned$$
i.e. $(A^{\Phi_\th(v)}u_\th)_\si=0$. Hence
$$0=\lan A^{\Phi_\th(v)}u_\th,w\ran=\lan A^{\Phi_\th(v)}w,u_\th\ran=|u_\th|^2$$
i.e. $u_\th=0$. Similarly, one can deduce that $B_{vw}^\th=0$.

Let $Y$ be a local smooth section of $T_\th M$ and $Z$ be a local smooth section of $T_\si M$, such that $Y_p=v$, $Z_p=w$. Lemma \ref{sec2} implies
$\lan B_{YZ},\Phi_\th(Y)\ran\equiv 0$, hence
\begin{equation}\label{sec21}\aligned
\lan (\n_v B)_{ww},\Phi_\th(v)\ran=&\lan (\n_w B)_{vw},\Phi_\th(v)\ran\\
=&\lan \n_w \lan B_{YZ},\Phi_\th(Y)\ran-\lan B_{vw},\n_w \Phi_\th(Y)\ran\\
&-\lan B_{\n_w Y,w},\Phi_\th(v)\ran-\lan B_{v,\n_w Z},\Phi_\th(v)\ran\\
=&-\lan B_{vw},\n_w \Phi_\th(Y)\ran-\lan B_{\n_w Y,w},\Phi_\th(v)\ran\\
:=&-I-II
\endaligned
\end{equation}
where
\begin{equation}\label{sec22}\aligned
I&=\sum_{\tau\in \text{Arg}^N}\lan B_{vw}^\tau,\n_w \Phi_\th(Y)\ran=\sum_{\tau\in \text{Arg}^N,\tau\neq \th}(S_{\th\tau}^N)_{\Phi_\th(v),B_{vw}^\tau}(w)\\
&=\sum_{\tau\in \text{Arg}^N,\tau\neq \th}\left(\k_{\th\tau}\lan B_{w,\Phi_\th^2(v)},B_{vw}^\tau\ran-\k_{\tau\th}\lan B_{w,\Phi_\tau(B_{vw}^\tau)},\Phi_\th(v)\ran\right)\\
&=-\sum_{\tau\in \text{Arg}^N,\tau\neq \th}\left(\k_{\th\tau}|B_{vw}^\tau|^2+\k_{\tau\th}\lan \Phi_\tau(B_{vw}^\tau),u_\tau\ran\right)\\
&=\sum_{\tau\in \text{Arg}^N,\tau\neq \th}\k_{\tau\th}\lan B_{vw}^\tau,\Phi_\tau(u_\tau)\ran-\sum_{\tau\in \text{Arg}^N,\tau\neq \th}\k_{\th\tau}|B_{vw}^\tau|^2
\endaligned
\end{equation}
and
\begin{equation}\label{sec23}\aligned
II&=\sum_{\tau\in \text{Arg}^T}\lan \n_w Y,u_\tau\ran=\sum_{\tau\in \text{Arg}^T,\tau\neq \th}(S_{\th\tau})_{v,u_\tau}(w)\\
&=\sum_{\tau\in \text{Arg}^T,\tau\neq \th}\left(\k_{\tau\th}\lan B_{vw},\Phi_\tau(u_\tau)\ran-\k_{\th\tau}\lan B_{w u_\tau},\Phi_\th(v)\ran\right)\\
&=\sum_{\tau\in \text{Arg}^T,\tau\neq \th}\k_{\tau\th}\lan B_{vw}^\tau,\Phi_\tau(u_\tau)\ran-\sum_{\tau\in \text{Arg}^T,\tau\neq \th}\k_{\th\tau}|u_\tau|^2.
\endaligned
\end{equation}
Substituting (\ref{sec22}) and (\ref{sec23}) into (\ref{sec21}) yields (\ref{sec20}).
\end{proof}

\bigskip

\begin{lem}\label{sec32}
If $\th\in \text{Arg}^T\cap \text{Arg}^N$ and $\th\equiv 0$ or $\pi/2$, then for any $v\in T_{p,\th}M$, $\nu\in \G(N_\th M)$ and $w\in T_p M$,
\begin{equation}\label{sec26}
\lan(\n_v B)_{ww},\nu\ran=-2\sum_{\si\in Arg^T}\k_{\si\th}\lan B_{vw}^\si,\Phi_\si(A^\nu w)_\si\ran.
\end{equation}
\end{lem}

\begin{proof}
Let $Y$ be a local section of $T_\th M$ and $Z$ be a local section of $TM$, such that $Y_p=v$ and $Z_p=w$, then Lemma \ref{sec25} tells us $\lan B_{YZ},\nu\ran\equiv 0$. Therefore
\begin{equation}\aligned\label{sec29}
\lan (\n_v B)_{ww},\nu\ran=&\lan (\n_w B)_{vw},\nu\ran\\
=&\n_w\lan B_{YZ},\nu\ran-\lan B_{vw},\n_w \nu\ran-\lan B_{\n_w Y,w},\nu\ran-\lan B_{v,\n_w Z},\nu\ran\\
=&-\lan B_{vw},\n_w \nu\ran-\lan B_{\n_w Y,w},\nu\ran\\
:=&-I-II
\endaligned
\end{equation}
where
\begin{equation}\label{sec27}\aligned
I&=\sum_{\si\in \text{Arg}^N,\si\neq \th} \lan B_{vw}^\si,\n_w \nu\ran=\sum_{\si\in \text{Arg}^N,\si\neq \th} (S_{\th\si}^N)_{\nu,B_{vw}^\si}(w)\\
&=-\sum_{\si\in \text{Arg}^N,\si\neq \th}\k_{\si\th}\lan B_{w,\Phi_\si(B_{vw}^\si)},\nu\ran=-\sum_{\si\in \text{Arg}^N} \k_{\si\th}\lan \Phi_\si(B_{vw}^\si),u_\si\ran\\
&=\sum_{\si\in \text{Arg}^T} \lan B_{vw}^\si,\Phi_\si (u_\si)\ran
\endaligned
\end{equation}
and
\begin{equation}\label{sec28}\aligned
II&=\sum_{\si\in \text{Arg}^T,\si\neq \th}\lan \n_w Y,u_\si\ran=\sum_{\si\in \text{Arg}^T,\si\neq \th} (S_{\th\si})_{v u_\si}(w)\\
&=\sum_{\si\in \text{Arg}^T,\si\neq \th} \k_{\si\th}\lan B_{wv},\Phi_\si(u_\si)\ran=\sum_{\si\in \text{Arg}^T} \k_{\si\th}\lan B_{vw}^\si,\Phi_\si(u_\si)\ran.
\endaligned
\end{equation}
Here $u_\si:=(A^\nu w)_\si$, and $u_\th=B_{vw}^\th=0$ is a direct corollary of Lemma \ref{sec25}. Substituting (\ref{sec27}) and (\ref{sec28}) into (\ref{sec29}), we arrive at
at (\ref{sec26}).

\end{proof}

\bigskip
\subsection{Vanishing theorems}
With the above lemmas, we can now derive  vanishing theorems for the second fundamental form of submanifolds with CJA.

\begin{thm}\label{tri}
Let $M^n$ be a submanifold of $\R^{n+m}$ with CJA relative to a fixed $m$-plane $Q_0$ ($M$ need not be complete), then

(i) If $g^T=g^N=1$, then $M$ has to be an affine linear subspace;

(ii) If $g^T=1, g^N=2, \pi/2\notin \Arg^T$ and $M$ has parallel mean curvature, then $M$ is affine linear;

(iii) If $g^T=2, g^N=1, \pi/2\notin \Arg^N$, and $M$ has parallel mean curvature, then $M$ is affine linear;

(iv) If $g^T=g^N=2$, $\Arg^N\neq \{0,\pi/2\}$, and $M$ is minimal, then $M$ is affine linear.
\end{thm}

\noindent \textbf{Remarks:}

\begin{itemize}
\item Let $S^1:=\{(x_1,x_2,x_3)\in \R^3: x_1^2+x_2^2=1,x_3=0\}$ be a circle whose tangent vectors are all orthogonal to the $x_3$-axis, then $S^1$ has CJA and $\Arg^T=\{\pi/2\}$,
$\Arg^N=\{\pi/2,0\}$. It is easy to check that $S^1$ has parallel mean curvature. Hence the condition '$\pi/2\notin \Arg^T$' cannot be dropped in (ii).

\item Let $S:=S^1\times \R$ be a circular cylinder, whose normal vectors are all orthogonal to the $x_3$-axis, then $S$ has CJA and $\Arg^N=\{\pi/2\}$, $\Arg^T=\{\pi/2,0\}$.
Its mean curvature vector field is parallel along $S$. Hence the condition '$\pi/2\notin \Arg^N$' cannot be dropped in (iii).

\item Let $S$ be a nontrivial minimal surface in $\R^3$, then $M:=S\times \R$ is a minimal submanifold in $\R^3\times \R^2=\R^5$. Then  $M$ has
CJA relative to $Q_0:=\R^2$, and $\Arg^N=\Arg^T=\{0,\pi/2\}$. Hence the condition '$\Arg^N\neq \{0,\pi/2\}$' cannot be dropped in (iv).
\end{itemize}

\begin{proof}

(i) Denote $g^T=g^N=\{\th\}$, then Lemma \ref{sec6} and \ref{sec25} tell us
$$\lan B_{vw},\nu\ran=0$$
for any $v,w\in T_{p,\th}M=T_p M$ and $\nu\in N_{p,\th}M=N_p M$. Hence $M$ is totally geodesic.
\bigskip

(ii) As shown in Section \ref{cja}, there exists $\th_0\neq 0,\pi/2$,
such that $\text{Arg}^T=\{\th_0\}$, $\text{Arg}^N=\{0,\th_0\}$.

By Lemma \ref{sec30},
\begin{equation}
\aligned
&\quad\ \k_{\th_0 0}\big(\lan B_{vv}^{0},B_{ww}^0\ran-|B_{vw}^0|^2\big)=\k_{\th_0 0}R_{\th_0 0}(v,w,v,w)\\
=&\sum_{\si\in \text{Arg}^N,\si\neq \th_0}\k_{\th_0\si}R_{\th_0\si}(v,w,v,w)=3\sum_{\si\in \text{Arg}^T,\si\neq \th_0}\k_{\th_0\si}U_{\th_0\si}(v,w,v,w)\\
=&0
\endaligned
\end{equation}
for any $v,w\in T_{p,\th_0}M=T_p M$. In conjunction with $\k_{\th_0 0}=\f{\sin 2\th_0}{\cos 2\th_0-1}\neq 0$, we have
$\lan B_{vv}^0,B_{ww}^0\ran=|B_{vw}^0|^2$. Substituting it into (\ref{sec9}) implies
\begin{equation}\aligned
\lan (\n_v B)_{ww},\Phi_{\th_0}(v)\ran&=(1/3)\k_{\th_0 0}\big(\lan B_{vv}^0,B_{ww}^0\ran+2|B_{vw}^0|^2\big)\\
&=\k_{\th_0 0}|B_{vw}^0|^2.
\endaligned
\end{equation}

Let $\{e_1,\cdots,e_n\}$ be an orthonormal basis of $T_{p,\th_0}M=T_p M$. Since $M$ has parallel mean curvature,
\begin{equation}
0=\sum_{i=1}^n \lan \n_v H,\Phi_{\th_0}(v)\ran=\sum_{i=1}^n \lan (\n_v B)_{e_ie_i},\Phi_{\th_0}(v)\ran=\k_{\th_0 0}\sum_{i=1}^n |B_{v e_i }^0|^2
\end{equation}
which forces $|B_{v e_i}^0|=0$ for any $1\leq i\leq n$. Thus $B_{vw}^0=0$ for any $v,w\in T_p M$. On the other hand, Lemma \ref{sec6} implies
$B_{vw}^{\th_0}=0$. Therefore $B\equiv 0$ on $M$.
\bigskip

(iii) Denote $\text{Arg}^N=\{\th_0\}$, $\text{Arg}^T=\{0,\th_0\}$ with $\th_0\neq 0,\pi/2$. Again applying Lemma \ref{sec30} gives
\begin{equation}
\k_{\th_0 0}U_{\th_0 0}(v,w,v,w)=(1/3)\sum_{\si\in \text{Arg}^N,\si\neq \th_0}\k_{\th_0\si}R_{\th_0\si}(v,w,v,w)=0
\end{equation}
i.e. $U_{\th 0}(v,w,v,w)=0$ for any $v,w\in T_{p,\th_0}M$. This means
\begin{equation}\aligned
0&=\lan (A^{\Phi_{\th_0}(v)}v)_0,(A^{\Phi_{\th_0}(w)}w)_0\ran-\lan (A^{\Phi_{\th_0}(w)}v)_0,(A^{\Phi_{\th_0}(v)}w)_0\ran\\
&=\big|(A^{\Phi_{\th_0}(v)}w)_0\big|^2.
\endaligned
\end{equation}
Since $\Phi_{\th_0}:T_{p,\th_0}M\ra N_{p,\th_0}M=N_p M$ is an isomorphism, $(A^\nu w)_0=0$ holds for every $\nu\in N_p M$. On the other hand,
$(A^\nu w)_{\th_0}=0$ is a direct corollary of Lemma \ref{sec6}. Thus $A^\nu w=0$ for every $w\in T_{p,\th_0}M$.

Let $\{e_1,\cdots,e_{m_{\th_0}}\}$ be an orthonormal basis of $T_{p,\th_0}M$, and $\{e_{m_{\th_0}+1},\cdots,e_n\}$ be an orthonormal basis of
$T_{p,0}M$. For any $v\in T_{p,\th_0}M$, by (\ref{sec9}) and (\ref{sec20}),
\begin{equation}
\lan (\n_v B)_{e_i e_i},\Phi_{\th_0}(v)\ran=\left\{\begin{array}{ll}
0 & \text{if } 1\leq i\leq m_{\th_0},\\
\k_{\th_0 0}\big|(A^{\Phi_{\th_0}(v)}e_i)_0\big|^2 & \text{if } m_{\th_0}+1\leq i\leq n.
\end{array}\right.
\end{equation}
Hence
\begin{equation}\aligned
0&=\lan \n_v H,\Phi_{\th_0}(v)\ran=\sum_{i=1}^n \lan (\n_v B)_{e_i e_i},\Phi_{\th_0}(v)\ran\\
&=\sum_{i=m_{\th_0}+1}^n \k_{\th_0 0}\big|(A^{\Phi_{\th_0}(v)}e_i)_0\big|^2
\endaligned
\end{equation}
and then $(A^\nu e_i)_0=0$ for any $\nu\in N_p M$. On the other hand,
$\lan A^\nu e_i,v\ran=\lan A^\nu v,e_i\ran=0$ holds for any $v\in T_{p,\th_0}M$. Therefore
$A^\nu e_i=0$ for each $m_{\th_0}+1\leq i\leq n$.

In summary, $A^\nu\equiv 0$ for any smooth section $\nu$ of $NM$ and then $M$ has to be affine linear.
\bigskip

(iv) Denote $\text{Arg}^N=\text{Arg}^T=\{\th_1,\th_2\}$. Without loss of generality one can assume $\th_1\in (0,\pi/2)$. Let $\{e_1,\cdots,e_m\}$ be an orthonormal
basis of $T_{p,\th_1}M$ and $\{e_{m+1},\cdots,e_{n}\}$ be an orthonormal basis of $T_{p,\th_2}M$. By Lemma \ref{sec30},
for any $1\leq i,j\leq m$,
\begin{equation*}\aligned
&\lan B_{e_ie_i}^{\th_2},B_{e_je_j}^{\th_2}\ran-|B_{e_ie_j}^{\th_2}|^2=R_{\th_1\th_2}(e_i,e_j,e_i,e_j)\\
=&3\ U_{\th_1\th_2}(e_i,e_j,e_i,e_j)=3\big|(A^{\Phi_{\th_1}(e_j)}e_i)\big|^2
\endaligned
\end{equation*}
i.e.
\begin{equation}
\lan B_{e_ie_i}^{\th_2},B_{e_je_j}^{\th_2}\ran=3\big|(A^{\Phi_{\th_1}(e_j)}e_i)\big|^2+|B_{e_ie_j}^{\th_2}|^2.
\end{equation}
 On the other hand, Lemma \ref{sec6} and Lemma \ref{sec25} tell us
$B_{e_i e_j}^{\th_2}=0$ for every $m+1\leq i,j\leq n$. Since $M$ is a minimal submanifold,
\begin{equation}\aligned
0&=|H^{\th_2}|^2=\big|\sum_{i=1}^n B_{e_ie_i}^{\th_2}\big|^2\\
&=\big|\sum_{i=1}^{m}B_{e_ie_i}^{\th_2}|^2=\sum_{i,j=1}^{m}\lan B_{e_ie_i}^{\th_2},B_{e_je_j}^{\th_2}\ran\\
&=\sum_{i,j=1}^{m}\big(3|(A^{\Phi_{\th_1}(e_j)}e_i)_{\th_2}|^2+|B_{e_ie_j}^{\th_2}|^2\big).
\endaligned
\end{equation}
Hence $(A^{\Phi_{\th_1}(e_j)}e_i)_{\th_2}=B_{e_ie_j}^{\th_2}=0$ for all $1\leq i,j\leq m$. In other words,
$B_{v_1v_2}^{\th_2}=0$ for any $v_1,v_2\in T_{p,\th_1}M$, and $B_{vw}^{\th_1}=0$ for any $v\in T_{p,\th_1}M$
and $w\in T_{p,\th_2}M$, which follows from the Weingarten equations.

If $\th_2\in (0,\pi/2)$, then similarly one can deduce that $B_{w_1w_2}^{\th_1}=0$ for any $w_1,w_2\in T_{p,\th_2}M$
and $B_{vw}^{\th_2}=0$ for any $v\in T_{p,\th_1}M$ and $w\in T_{p,\th_2}M$. In conjunction with $B_{v_1v_2}^{\th_1}=0$
for any $v_1,v_2\in T_{p,\th_1}M$ and $B_{w_1w_2}^{\th_2}=0$ for any $w_1,w_2\in T_{p,\th_2}M$, we have $B\equiv 0$ on $M$
and $M$ has to be totally geodesic.

If $\th_2=0$ or $\pi/2$, then (\ref{sec9}) implies
\begin{equation}
\lan (\n_v B)_{e_ie_i},\Phi_{\th_1}(v)\ran=(1/3)\k_{\th_1\th_2}\left(\lan B_{e_ie_i}^{\th_2},B_{vv}^{\th_2}\ran+2|B_{e_i v}^{\th_2}|^2\right)=0
\end{equation}
for any $v\in T_{p,\th_1}M$ and each $1\leq i\leq m$. Since $U_{\th_1\th_2}(v_1,v_2,v_1,v_2)=0$ for any $v_1,v_2\in T_{p,\th_1}M$, (\ref{sec20}) tells us
\begin{equation}
\lan (\n_v B)_{e_ie_i},\Phi_{\th_1}(v)\ran=\k_{\th_1\th_2}|B_{ve_i}^{\th_2}|^2+\k_{\th_1\th_2}|(A^{\Phi_{\th_1}(v)}e_i)_{\th_2}|^2.
\end{equation}
for each $m+1\leq i\leq n$. Thus
\begin{equation}\aligned
0=\lan \n_v H,\Phi_\th(v)\ran&=\sum_{i=1}^n \lan (\n_v B)_{e_ie_i},\Phi_\th(v)\ran\\
&=\sum_{i=m+1}^n \k_{\th_1\th_2}\left(|B_{ve_i}^{\th_2}|^2+|(A^{\Phi_{\th_1}(v)}e_i)_{\th_2}|^2\right),
\endaligned
\end{equation}
which forces $B_{ve_i}^{\th_2}=(A^{\Phi_{\th_1}(v)}e_i)_{\th_2}=0$ for each $m+1\leq i\leq n$. In other words,
$B_{vw}^{\th_2}=0$ for any $v\in T_{p,\th_1}M$ and $w\in T_{p,\th_2}M$, and $B_{w_1w_2}^{\th_1}=0$ for any $w_1,w_2\in
T_{p,\th_2}M$. Therefore $B\equiv 0$ on $M$ and $M$ has to be affine linear.

\end{proof}

Let $f:D\subset \R^n\ra \R^m$ be a smooth vector-valued function, then for any $p\in M:=\text{graph }f$, any Jordan angle between $N_p M$
and the coordinate $m$-plane takes values in $[0,\pi/2)$ (see \cite{x-y1}). Hence Theorem \ref{tri} implies:

\begin{cor}
Let $D$ be an open domain of $\R^n$ and $f:D\ra \R^m$. If $M=\text{graph }f$ is a minimal submanifold with CJA relative
to the coordinate $m$-plane, and $g^N,g^T\leq 2$, then $M$ has to be an affine $n$-plane.

\end{cor}

This is the Theorem 1.1 mentioned in \S 1.5.

\bigskip\bigskip

\Section{Coassociative submanifolds with CJA}{Coassociative submanifolds with CJA}\label{co-cja}

\subsection{Associative subspace of $\Im \O$}
Let $\Bbb{O}$ denote the octonions, which is an $8$-dimensional normed algebra over $\R$ with multiplicative unit $1$. More precisely,
$\Bbb{O}$ is equipped with an inner product $\lan\cdot,\cdot\ran$, whose associated norm $|\cdot|$ satisfies
\begin{equation}\label{norm}
 |xy|=|x||y|
\end{equation}
for any $x,y\in \Bbb{O}$.
Denote by $\Re \O$ the 1-dimensional subspace spanned by $1$, and by $\Im \O$ the orthogonal complement of $\Re\O$.
Then every $x\in \O$ has a unique decomposition
$$x=\Re x+\Im x$$
with $\Re x\in \Re \O, \Im x\in \Im \O$. The conjugation
of $x$ is defined by
\begin{equation}\label{conjugation}
 \bar{x}=\Re x-\Im x.
\end{equation}
For $w\in \O$, let $R_w$ ($L_w$) denote the linear operator of right (left) multiplication by $w$, respectively. With the aid of
(\ref{norm}) and (\ref{conjugation}), one can easily deduce the following fundamental formulas (see  Appendix IV.A of \cite{h-l}):
\begin{equation}\label{o1}
 \lan R_w x,R_w y\ran=\lan x,y\ran |w|^2,\qquad  \lan L_w x,L_w y\ran=\lan x,y\ran |w|^2,
\end{equation}
\begin{equation}\label{o2}
\lan x,R_w y\ran=\lan R_{\bar{w}}x,y\ran,\qquad \lan x,L_w y\ran=\lan L_{\bar{w}}x,y\ran,
\end{equation}
\begin{equation}\label{o3}
\bar{\bar{x}}=x,\ \ol{xy}=\bar{y}\bar{x},\qquad x\bar{x}=|x|^2,\ \lan x,y\ran=\Re x\bar{y}.
\end{equation}


Let $P$ be a $3$-dimensional real subspace of $\Im\O$, if $A:=\Re \O\oplus P$ is a quarternion subalgebra
of $\O$ (i.e. $A$ is isomorphic to $\H$), then $P$ is said to be \textit{associative}.

\begin{lem}\label{o4}
Let $P$ be an associative subspace of $\Im \O$ and $x,y$ be unit elements in $P$ that are orthogonal to each other,
then $\{x,y,z:=xy\}$ is an orthonormal basis of $P$, and
\begin{equation}
xy=-yx=z,\quad yz=-zy=x,\quad zx=-xz=y.
\end{equation}
Conversely, if $\{x,y,z\}$ is an orthonormal basis of an associative subspace $P$, then $z=xy$ or $-xy$.
\end{lem}

\begin{proof}
 Since $\Re \O\oplus P$ is a subalgebra of $\O$, $xy\in \Re\O\oplus P$. By (\ref{conjugation}) and (\ref{o3}),
$$\Re(xy)=-\Re(x\bar{y})=-\lan x,y\ran=0,$$
i.e. $xy\in P$. Applying (\ref{o1}) and (\ref{norm}) gives
$$\aligned
\lan xy,x\ran&=\lan L_x y,L_x 1\ran=\lan y,1\ran|x|^2=0,\\
\lan xy,y\ran&=\lan R_y x,R_y 1\ran=\lan x,1\ran|y|^2=0,\\
|xy|&=|x||y|=1.
\endaligned$$
Hence $\{x,y,z:=xy\}$ is an orthonormal basis of $P$.

Similarly, one can show $yx$ is also a unit element in $P$ orthogonal to $\text{span}\{x,y\}$,
hence $yx=z$ or $-z$. If $yx=z$, then
\begin{equation}\label{o6}\aligned
&(x+y)(x-y)=x^2-y^2+yx-xy\\
=&-x\bar{x}+y\bar{y}+z-z=-|x|^2+|y|^2\\
=&0.
\endaligned
\end{equation}
On the other hand, since $x$ and $y$ are linearly independent, $x+y,x-y\neq 0$ and it follows from (\ref{norm}) that
$|(x+y)(x-y)|=|x+y||x-y|\neq 0$, which contradicts  (\ref{o6}). Hence $yx=-z$ and it follows that
$$\aligned
yz=&y(-yx)=-y^2 x\\
=&y\bar{y}x=|y|^2 x=x.
\endaligned$$
Similarly one can prove $zy=-x$ and $zx=-xz=y$.

Conversely, if $\{x,y,z\}$ is an orthonormal basis of $P$, then $z$ and $xy$ are both unit elements orthogonal to $\text{span}\{x,y\}$,
which implies $z=xy$ or $-xy$.

\end{proof}

\begin{lem}\label{mult}
 Let $A$ be a quarternion subalgebra of $\O$, $\ep\in A^\bot$ with $|\ep|=1$, then $A\ep\bot A$, $\O=A\oplus A\ep$
and
\begin{equation}
 (x+y\ep)(v+w\ep)=(xv-\bar{w}y)+(wx+y\bar{v})\ep
\end{equation}
for any $x,y,v,w\in A$.

\end{lem}

\begin{proof}
 The lemma immediately follows from Lemma A.8 in \cite{h-l}.
\end{proof}

\subsection{Jordan angles between associative subspaces}
Now we explore the Jordan angles between an associative subspace $P$ and $\Im \H$.

\textbf{Case I.} $0\in \text{Arg}(P,\Im\H)$ and $m_0\geq 2$. This means there exist $2$ unit elements $a,b\in P\cap \Im\H$
that are orthogonal to each other, then it follows from Lemma \ref{o4} that $\{a,b,ab\}$ is an orthonormal basis
of $P\cap \Im \H$. Hence $P=\Im \H$ and $\text{Arg}(P,\Im\H)=\{0\}$.

\textbf{Case II.} $\pi/2\in \text{Arg}(P,\Im\H)$ and $m_{\pi/2}\geq 2$. Then there exists $2$ unit elements $ae,be\in P\cap (\Im \H)^\bot=
P\cap \H e$ that are orthogonal to each other. By Lemma \ref{o4}, $(ae)(be)=-\bar{b}a$ is a unit vector in
$P$, and $-\bar{b}a\in \H\cap \Im\O=\Im \H$. Hence $\text{Arg}(P,\Im \H)=\{0,\pi/2\}$,
$m_0=1$, $m_{\pi/2}=2$, and $P$ is spanned by $ae,be$ and $-\bar{b}a$, which are the angle directions
of $P$ relative to $\Im \H$.

\textbf{Case III.} $m_0\leq 1$ and $m_{\pi/2}\leq 1$. (Note that $m_0=0$ ($m_{\pi/2}=0$) means $0\notin \text{Arg}(P,\Im\H)$ (
$\pi/2\notin  \text{Arg}(P,\Im\H)$), respectively.) Firstly, we claim $m_0+m_{\pi/2}\leq 1$. If not, there exist
unit elements $a\in P\cap \Im\H$ and $be\in P\cap \H e$; by Lemma \ref{o4}, $P$ is spanned by $a, be$ and
$a(be)=(ba)e\in \H e$; hence $m_{\pi/2}=2$, contradicting $m_{\pi/2}\leq 1$.

Hence there exist mutually orthogonal elements $x_1,x_2\in P$ that are unit angle directions of $P$ relative
to $\Im \H$ associated to $\th_1,\th_2\in \text{Arg}(P,\Im \H)\cap (0,\pi/2)$, respectively.
More precisely,
\begin{equation}
 (\mc{P}\circ \mc{P}_0)x_\a=\cos^2\th_\a x_\a \qquad \forall \a=1,2.
\end{equation}
Here $\mc{P}_0$ denotes the orthogonal projection of $\Im\O$ onto $\Im\H$ and $\mc{P}$ denotes
the orthogonal projection of $\Im\O$ onto $P$. As in Section \ref{cja}, we denote by $\mc{P}_0^\bot$
the orthogonal projection of $\Im\O$ onto $\H e=(\Im \H)^\bot$, then
\begin{equation}
 \aligned
x_\a&=\mc{P}_0 x_\a+\mc{P}_0^\bot x_\a\\
&=\cos\th_\a a_\a+\sin\th_\a y_\a
\endaligned
\end{equation}
with $a_\a:=\sec\th_\a\ \mc{P}_0x_\a\in \Im\H$ and $y_\a:=\csc\th_\a\ \mc{P}^\bot x_\a\in \H e$,
satisfying $|a_\a|=|y_\a|=1$ for each $\a=1,2$. Let $\ep$ be the unique element in $\O$ satisfying
$y_1=a_1\ep$, then for every $c\in \H$,
$$\aligned
\lan \ep,c\ran&=\lan L_{a_1}\ep,L_{a_1}c\ran=\lan a_1\ep,a_1 c\ran\\
&=\lan y_1,a_1 c\ran=0,
\endaligned$$
which implies $\ep\in \H e$. And $|\ep|=1$ directly follows from $y_1=a_1\ep$ and $|y_1|=|a_1|=1$.
Similarly, one can prove that there exists a unique $b\in \H$ which satisfies $y_2=b\ep$, and moreover
$|b|=1$.

Let $x_3:=x_1x_2$, then Lemma \ref{mult} enables us to obtain
\begin{equation}\label{o7}\aligned
x_3=&(\cos\th_1 a_1+\sin\th_1 a_1\ep)(\cos\th_2 a_2+\sin\th_2 b\ep)\\
=&(\cos\th_1\cos\th_2 a_1a_2-\sin\th_1\sin\th_2 \bar{b}a_1)\\
&+(\cos\th_1\sin\th_2 ba_1+\sin\th_1\cos\th_2 a_1\bar{a}_2)\ep.
\endaligned
\end{equation}

By Lemma \ref{o4}, $\{x_1,x_2,x_3\}$ is an orthonormal basis of $P$, thus for each $\a=1,2$,
\begin{equation}\label{o5}\aligned
0&=\cos^2 \th_\a\lan x_\a,x_3\ran=\lan (\mc{P}\circ \mc{P}_0)x_\a,x_3\ran\\
&=\lan \mc{P}_0 x_\a,x_3\ran=\lan \mc{P}_0 x_\a,\mc{P}_0 x_3\ran.
\endaligned
\end{equation}
When $\a=1$, the above equation gives
$$\aligned
0&=\lan \mc{P}_0 x_1,\mc{P}_0 x_3\ran=\lan \cos\th_1 a_1,\cos\th_1 \cos\th_2 a_1 a_2-\sin\th_1\sin\th_2 \bar{b}a_1\ran\\
&=\cos^2\th_1 \cos\th_2\lan a_1,a_1 a_2\ran-\cos\th_1\sin\th_1\sin\th_2\lan a_1,\bar{b}a_1\ran\\
&=\cos^2\th_1\cos\th_2\lan 1,a_2\ran-\cos\th_1\sin\th_1\sin\th_2\lan 1,\bar{b}\ran\\
&=-\cos\th_1\sin\th_1\sin\th_2\lan \bar{b},1\ran.
\endaligned$$
In conjunction with $\th_1,\th_2\in (0,\pi/2)$ we have $\lan\bar{b},1\ran=0$, therefore $b\in \Im\H$. Letting $\a=2$ in (\ref{o5})
yields
$$\aligned
0&=\lan \mc{P}_0 x_2,\mc{P}_0 x_3\ran=\lan \cos\th_2 a_2,\cos\th_1\cos\th_2 a_1a_2-\sin\th_1\sin\th_2 \bar{b}a_1\ran\\
&=\cos\th_1\cos^2\th_2\lan a_2,a_1a_2\ran-\cos\th_2\sin\th_1\sin\th_2\lan a_2,\bar{b}a_1\ran\\
&=-\cos\th_2\sin\th_1\sin\th_2\lan a_2,\bar{b}a_1\ran
\endaligned$$
and moreover
$$\aligned
0&=\lan a_2,\bar{b}a_1\ran=\lan a_2,R_{a_1}\bar{b}\ran=\lan R_{\bar{a}_1}a_2,\bar{b}\ran\\
&=\lan a_2\bar{a}_1,\bar{b}\ran=\lan -a_2a_1,-b\ran=\lan a_2a_1,b\ran.
\endaligned$$
Observing that $a_1,a_2$ and $a_2a_1$ form an orthonormal basis of $\Im\H$, we have $b\in \text{span}\{a_1,a_2\}$.

By the definition of angle directions,
\begin{equation}\aligned
&\lan \mc{P}_0^\bot x_1,\mc{P}_0^\bot x_2\ran=\lan \mc{P}_0^\bot x_1,x_2\ran=\lan (\mc{P}\circ \mc{P}_0^\bot)x_1,x_2\ran\\
=&\lan \mc{P}(x_1-\mc{P}_0 x_1),x_2\ran=\lan x_1,x_2\ran-\lan (\mc{P}\circ \mc{P}_0)x_1,x_2\ran\\
=&\lan x_1,x_2\ran-\cos^2\th_1\lan x_1,x_2\ran=0,
\endaligned
\end{equation}
which implies
$$\aligned
0&=\lan \sin\th_1 y_1,\sin\th_2 y_2\ran=\sin\th_1\sin\th_2\lan a_1\ep,b\ep\ran\\
&=\sin\th_1\sin\th_2\lan a_1,b\ran
\endaligned$$
i.e. $\lan a_1,b\ran=0$. Therefore $b=a_2$ or $-a_2$.

If $b=a_2$, then (\ref{o7}) shows
\begin{equation}\aligned
x_3&=(\cos\th_1\cos\th_2 a_1a_2-\sin\th_1\sin\th_2 \bar{a}_2 a_1)+(\cos\th_1\sin\th_2a_2a_1+\sin\th_1\cos\th_2a_1\bar{a}_2)\ep\\
&=(\cos\th_1\cos\th_2-\sin\th_1\sin\th_2)a_3-(\cos\th_1\sin\th_2+\sin\th_1\cos\th_2)a_3\ep\\
&=\cos(\th_1+\th_2)a_3-\sin(\th_1+\th_2)a_3\ep.
\endaligned
\end{equation}
Noting that $x_3$ is also an angle direction of $P$ relative to $\Im\H$, $\th_3:=\arccos |\cos(\th_1+\th_2)|\in \text{Arg}(P,\Im\H)$.
In other words,
\begin{equation*}
\th_3=\left\{\begin{array}{cc}
\th_1+\th_2 & \text{if }\th_1+\th_2\leq \pi/2,\\
\pi-(\th_1+\th_2) & \text{if }\th_1+\th_2>\pi/2.
\end{array}\right.
\end{equation*}

Otherwise, $b=-a_2$ and (\ref{o7}) gives
\begin{equation}\aligned
x_3&=(\cos\th_1\cos\th_2 a_1a_2-\sin\th_1\sin\th_2 (-\bar{a}_2) a_1)+(\cos\th_1\sin\th_2(-a_2a_1)+\sin\th_1\cos\th_2a_1\bar{a}_2)\ep\\
&=(\cos\th_1\cos\th_2+\sin\th_1\sin\th_2)a_3-(-\cos\th_1\sin\th_2+\sin\th_1\cos\th_2)a_3\ep\\
&=\cos(\th_1-\th_2)a_3-\sin(\th_1-\th_2)a_3\ep,
\endaligned
\end{equation}
which implies $\th_3:=\arccos|\cos(\th_1-\th_2)|=|\th_1-\th_2|\in \text{Arg}(P,\Im\H)$. Without loss of generality, one can assume
$\th_1\geq \th_2$, then $\th_3=\th_1-\th_2$. Now we put
$$\aligned
&\th'_1:=\th_2,\quad\th'_2:=\th_3,\quad\th'_3:=\th_1,\\
&a'_1:=a_2,\quad a'_2:=a_3,\quad a'_3:=a_1,\\
&x'_1:=x_2,\quad x'_2:=x_3,\quad x'_3:=x_1
\endaligned$$
and $\ep':=-\ep$, then
\begin{equation}
\aligned
x'_1&=\cos\th'_1 a'_1+\sin\th'_1 a'_1\ep',\\
x'_2&=\cos\th'_2 a'_2+\sin\th'_2 a'_2\ep',\\
x'_3&=\cos\th'_3 a'_3-\sin\th'_3 a'_3\ep',
\endaligned
\end{equation}
which satisfy $\th'_3=\th'_1+\th'_2$, $a'_3=a'_1a'_2$ and $x'_3=x'_1x'_2$.

Altogether, we have shown
\bigskip
\begin{pro}\label{o8}
Let $P$ be an associative subspace of $\Im\O$, and $0\leq \th_1\leq \th_2\leq \th_3\leq \pi/2$ be the Jordan angles between $P$ and $\Im\H$, then
\begin{equation}
\th_3=\left\{\begin{array}{cc}
\th_1+\th_2 & \text{if }\th_1+\th_2\leq \pi/2,\\
\pi-(\th_1+\th_2) & \text{if }\th_1+\th_2>\pi/2.
\end{array}\right.
\end{equation}
Moreover, there exist an orthonormal basis $\{a_1,a_2,a_3\}$ of $\Im\H$ satisfying $a_3=a_1a_2$, and a unit element $\ep\in \H e$, such that
\begin{equation}\aligned
x_1:=&\cos\th_1 a_1+\sin\th_1 a_1\ep,\\
x_2:=&\cos\th_2 a_2+\sin\th_2 a_2\ep,\\
x_3:=&\cos(\th_1+\th_2)a_3-\sin(\th_1+\th_2) a_3\ep
\endaligned
\end{equation}
are unit angle directions of $P$ relative to $\Im \H$, and $x_3=x_1x_2$.
\end{pro}

\subsection{On the second fundamental form of coassociative submanifolds}
Let $M$ be a 4-dimensional submanifold in $\Im\O$. If the normal space at every point of $M$ is associative,
then we call $M$ a \textit{coassociative submanifold} (see \cite{h-l}).  Let $p$ be an arbitrary point of $M$, denote by
$0\leq \th_1\leq \th_2\leq \th_3\leq \pi/2$ the Jordan angles between $N_{p}M$ (an associative subspace) and $\Im\H$, then
by Proposition \ref{o8},
\begin{equation}\label{o9}
 \th_3=\left\{\begin{array}{cc}
\th_1+\th_2 & \text{if }\th_1+\th_2\leq \pi/2,\\
\pi-(\th_1+\th_2) & \text{if }\th_1+\th_2>\pi/2.
\end{array}\right.
\end{equation}
Denote $\{a_1,a_2,a_3\}$ to be the orthonormal basis of $\Im\H$ satisfying $a_3=a_1a_2$ and $\ep$
to be the unit element in $\H e$, such that
\begin{equation}
 \aligned
\nu_1:=&\cos\th_1 a_1+\sin\th_1 a_1\ep,\\
\nu_2:=&\cos\th_2 a_2+\sin\th_2 a_2\ep,\\
\nu_3:=&\cos(\th_1+\th_2)a_3-\sin(\th_1+\th_2) a_3\ep
\endaligned
\end{equation}
are all unit angle directions of $N_{p}M$ relative to $\Im\H$, and $\nu_3=\nu_1\nu_2$. Denote
\begin{equation}
 \aligned
&e_1:=-\nu_1\ep=\sin\th_1 a_1-\cos\th_1 a_1\ep,\\
&e_2:=-\nu_2\ep=\sin\th_2 a_2-\cos\th_2 a_2\ep,\\
&e_3:=-\nu_3\ep=-\sin(\th_1+\th_2)a_3-\cos(\th_1+\th_2)a_3\ep,\\
&e_4:=\ep,
\endaligned
\end{equation}
then it is easy to check that $\lan e_i,\nu_\a\ran=0$ and $\lan e_i,e_j\ran=\de_{ij}$ for each
$1\leq i,j\leq 4$ and $1\leq \a\leq 3$. Hence $\{e_1,e_2,e_3,e_4\}$ is an orthonormal basis of $T_{p}M$.
Whenever $\th_\a\in (0,\pi/2)$, let $\Phi_{p,\th_\a}$ denote the isometric automorphism of $R_{p,\th_\a}M:=N_{p,\th_\a}M\oplus T_{p,\th_\a}M$
 as in \S \ref{1.2}, then it follows from (\ref{Phi}) that
$$\sec\th_1 \mc{P}_0^\bot e_1=\cos\th_1 e_1-\sin \th_1 \Phi_{p,\th_1}(e_1).$$
Hence
$$\aligned
\Phi_{p,\th_1}(e_1)&=\cot \th_1 e_1-\sec\th_1\csc\th_1 \mc{P}_0^\bot e_1\\
&=\cot\th_1(\sin\th_1 a_1-\cos\th_1 a_1\ep)-\sec\th_1\csc\th_1(-\cos\th_1 a_1\ep)\\
&=\cos\th_1 a_1+\sin\th_1 a_1\ep\\
&=\nu_1
\endaligned$$
and similarly $\Phi_{p,\th_2}(e_2)=\nu_2$; in conjunction with (\ref{o9}),
$$\aligned
\Phi_{p,\th_3}(e_3)&=\cot\th_3 e_3-\sec\th_3\csc\th_3 \mc{P}_0^\bot e_3\\
&=-\cos\th_3 a_3+\text{sgn}\big(\cos(\th_1+\th_2)\big)\sin\th_3 a_3\ep\\
&=\left\{\begin{array}{cc}
          -\nu_3 & \text{if }\th_1+\th_2<\pi/2,\\
            \nu_3 & \text{if }\th_1+\th_2>\pi/2.
         \end{array}
\right.
\endaligned$$
In summary we get a proposition as follows.
\bigskip

\begin{pro}\label{coass}
 Let $M$ be a coassociative submanifold in $\Im\O$, $p\in M$ and $0\leq \th_1\leq \th_2\leq \th_3\leq \pi/2$
be the Jordan angles between $N_{p}M$ and $\Im \H$, then there exist an orthonormal basis $\{\nu_1,\nu_2,\nu_3\}$
of $N_{p}M$ and an orthonormal basis $\{e_1,e_2,e_3,e_4\}$ of $T_{p}M$, such that

(i) For each $1\leq \a\leq 3$, $\nu_\a$ ($e_\a$) is an angle direction of $N_{p}M$ ($T_{p}M$) relative
to $\Im \H$ (or $\H e$), corresponding to the Jordan angle $\th_\a$;

(ii) $e_4\in \H e$;

(iii) $e_\a=-\nu_\a e_4$ for each $1\leq \a\leq 3$;

(iv) $\Phi_{p,\th_\a}(e_\a)=\nu_\a$ for any $1\leq \a\leq 2$ satisfying $\th_\a\in (0,\pi/2)$;

(v) $\Phi_{p,\th_3}(e_3)=\left\{\begin{array}{cc}
                                   -\nu_3 & \text{if }\th_1+\th_2< \pi/2,\\
                                    \nu_3 & \text{if }\th_1+\th_2>\pi/2,\\
                                        0 & \text{if }\th_1+\th_2=\pi/2.
                                  \end{array}
\right.$

\end{pro}
\noindent \textbf{Remark. }Here we additionally define $\Phi_{p,0}=\Phi_{p,\pi/2}=0$, as in \S \ref{2.2}.

\bigskip

Now we extend $\{\nu_1,\nu_2,\nu_3\}$ as an orthonormal normal frame field on $U$, a neighborhood of $p$, such
that $\n_v \nu_\a=0$ for every $v\in T_{p}M$. Lemma \ref{o4} implies $\nu_3(q)=\nu_1(q)\nu_2(q)$ or $-\nu_1(q)\nu_2(q)$
for an arbitrary $q\in U$. Due to the continuity, $\nu_3=\nu_1\nu_2$ on $U$ and differentiating both sides with respect to
$e_i\in T_{p}M$ gives
$$\aligned
-h_{3,ij}e_j&=\ol{\n}_{e_i}\nu_3=\ol{\n}_{e_i}(\nu_1\nu_2)\\
&=(\ol{\n}_{e_i}\nu_1)\nu_2+\nu_1(\ol{\n}_{e_i}\nu_2)\\
&=-h_{1,ij}e_j\nu_2-h_{2,ij}\nu_1 e_j,
\endaligned$$
i.e.
\begin{equation}\label{o10}
 h_{3,ij}e_j=h_{1,ij}e_j\nu_2+h_{2,ij}\nu_1e_j.
\end{equation}
With the aid of Lemma \ref{o4}, Lemma \ref{mult} and Proposition \ref{coass}, a straightforward calculation shows
\begin{equation}\label{o11}\aligned
\text{LHS of (\ref{o10})}&=h_{3,i\a}e_\a+h_{3,i4}e_4\\
&=-h_{3,i\a}\nu_\a e_4+h_{3,i4}e_4
\endaligned
\end{equation}
and
\begin{equation}\label{o12}\aligned
 \text{RHS of (\ref{o10})}=&h_{1,i\a}e_\a\nu_2+h_{1,i4}e_4\nu_2+h_{2,i\a}\nu_1 e_\a+h_{2,i4}\nu_1 e_4\\
=&-h_{1,i\a}(\nu_\a e_4)\nu_2-h_{1,i4}\nu_2 e_4-h_{2,i\a}\nu_1(\nu_\a e_4)+h_{2,i4}\nu_1 e_4\\
=&h_{1,i\a}(\nu_\a \nu_2)e_4-h_{1,i4}\nu_2 e_4-h_{2,i\a}(\nu_\a\nu_1)e_4+h_{2,i4}\nu_1 e_4\\
=&h_{1,i1}\nu_3 e_4-h_{1,i2}e_4-h_{1,i3}\nu_1e_4-h_{1,i4}\nu_2 e_4\\
&+h_{2,i1}e_4+h_{2,i2}\nu_3 e_4-h_{2,i3}\nu_2e_4+h_{2,i4}\nu_1 e_4\\
=&(-h_{1,i2}+h_{2,i1})e_4+(-h_{1,i3}+h_{2,i4})\nu_1 e_4\\
&+(-h_{1,i4}-h_{2,i3})\nu_2 e_4+(h_{1,i1}+h_{2,i2})\nu_3 e_4.
\endaligned
\end{equation}
Comparing with (\ref{o11}) and (\ref{o12}), we arrive at the following conclusion.
\bigskip

\begin{pro}\label{co_sec11}
 Let $M$ be a coassociative submanifold in $\Im \O$, $p\in M$. Let  $\{e_i:1\leq i\leq 4\}$
and $\{\nu_\a:1\leq \a\leq 3\}$ be as in Proposition \ref{coass}. Then for each $1\leq i\leq 4$,
\begin{eqnarray}
         &&h_{3,i1}=h_{1,i3}-h_{2,i4},\label{co1}\\
         &&h_{3,i2}=h_{1,i4}+h_{2,i3},\label{co2}\\
         &&h_{3,i3}=-h_{1,i1}-h_{2,i2},\label{co3}\\
         &&h_{3,i4}=-h_{1,i2}+h_{2,i1}.\label{co4}
\end{eqnarray}
Here $\{h_{ij}^\a:=\lan B_{e_i e_j},\nu_\a\ran(p):1\leq i,j\leq 4,1\leq \a\leq 3\}$ are the coefficients of the
second fundamental form at $p$.

\end{pro}
\bigskip

\subsection{The characterization of the Lawson-Osserman's cone}
Now we additionaly assume $M$ has CJA relative to $\Im \H$. Let $p_0\in M$,
$0\leq \th_1\leq \th_2\leq \th_3\leq \pi/2$ be the Jordan angles between $N_{p_0}M$ and $\Im\H$,
$\{\nu_1,\nu_2,\nu_3\}$ be the orthonormal basis of $N_{p_0}M$ and $\{e_1,e_2,e_3,e_4\}$ be the orthonormal
basis of $T_{p_0}M$, satisfying the properties in Proposition \ref{coass}. Then (\ref{sec1}) and Lemma \ref{sec25}
implies
\begin{equation}\label{co_sec1}
 h_{\a,\a i}=0 \qquad \forall 1\leq \a\leq 3,1\leq i\leq 4
\end{equation}
and substituting it into (\ref{co1})-(\ref{co4}) gives
\begin{eqnarray}
 &&h_{1,23}=h_{2,31}=h_{3,12};\label{co5}\\
&&h_{1,22}=-h_{3,24},\quad h_{1,33}=h_{2,34},\quad h_{1,44}=-h_{2,34}+h_{3,24};\label{co6}\\
&&h_{2,33}=-h_{1,34},\quad h_{2,11}=h_{3,14},\quad h_{2,44}=-h_{3,14}+h_{1,34};\label{co7}\\
&&h_{3,11}=-h_{2,14},\quad h_{3,22}=h_{1,24},\quad h_{3,44}=-h_{1,24}+h_{2,14}.\label{co8}
\end{eqnarray}
Furthermore, applying Lemma \ref{sec30} and \ref{sec31} yields the following propositions.
\bigskip

\begin{pro}\label{p1}
 Let $M$ be a coassociative submanifold in $\Im \O$, with CJA relative to $\Im \H$. If
$g^N\leq 2$, $\pi/2\notin \text{Arg}^N$ and $\text{Arg}^N\neq \{\arccos (\sqrt{6}/6),\arccos (2/3)\}$, then $M$ has to be affine linear.
\end{pro}

\begin{proof}
 Let $p_0$ be an arbitrary point in $M$, and the notations $\th_\a$, $\nu_\a$, $e_i$, $h_{\a,ij}$ are same as above.

\textbf{Case I.} $\th_1=0$ and $\th_2=\th_3<\pi/2$. Then $g^T=g^N\leq 2$ and the equality holds
if and only if $\th_2\neq 0$. It is well-known that coassociative submanfolds are absolutely area minimizing (see \cite{h-l}
\S IV.2.B). By Theorem \ref{tri}, $M$ has to be an open set of an affine 4-plane.

\textbf{Case II.} $\th_1=\th_2\in (0,\pi/4)\cup (\pi/4,\pi/3)$ and $\th_3=\left\{\begin{array}{cc}
                                                           2\th_1  & \text{if }\th_1< \pi/4,\\
                                                           \pi-2\th_1 & \text{if }\th_1>\pi/4.
                                                          \end{array}\right.$
Denote $\th:=\th_1$, then $\text{Arg}^N=\{\th,\th_3\}$, $\text{Arg}^T=\{0,\th,\th_3\}$;
$T_{p_0,\th}M=\text{span}\{e_1,e_2\}$ and $N_{p_0,\th}M=\text{span}\{\nu_1,\nu_2\}$
with $\nu_\a=\Phi_\th(e_\a)$ for each $1\leq \a\leq 2$;
$T_{p_0,\th_3}M=\text{span}\{e_3\}$, $N_{p_0,\th_3}M=\text{span}\{\nu_3\}$ and
$$\Phi_{\th_3}(e_3)=\left\{\begin{array}{cc}
                           \nu_3 & \text{if }\th_1>\pi/4,\\
                          -\nu_3 & \text{if }\th_1<\pi/4;
                          \end{array}
\right.$$ $T_{p_0,0}=\text{span}\{e_4\}$. Lemma \ref{sec6} implies
\begin{equation}\label{co_sec7}
 h_{1,22}=h_{2,11}=0.
\end{equation}
Substituting the above equation into (\ref{co6}) and (\ref{co7}), we get
\begin{equation}\label{co_sec2}
 h_{3,24}=h_{3,14}=0.
\end{equation}
Applying Lemma \ref{sec2} gives
\begin{equation}
 0=h_{1,23}+h_{2,13}=h_{1,24}+h_{2,14}.
\end{equation}
In conjunction with (\ref{co5}), we have
\begin{equation}\label{co_sec3}
 h_{1,23}=h_{2,31}=h_{3,12}=0.
\end{equation}
Let $R_{\th\si}$ and $U_{\th\si}$ be tensors of type $(4,0)$, defined in (\ref{R}) and (\ref{U}), respectively. Then
\begin{equation}\aligned
 R_{\th\th_3}(e_1,e_2,e_1,e_2)&=\lan B_{e_1e_1}^{\th_3},B_{e_2e_2}^{\th_3}\ran-\lan B_{e_1e_2}^{\th_3},B_{e_2e_1}^{\th_3}\ran\\
&=h_{3,11}h_{3,22}-h_{3,12}h_{3,21}\\
&=-h_{2,14}h_{1,24}=h_{1,24}^2,
\endaligned
\end{equation}
\begin{equation}\aligned
U_{\th\th_3}(e_1,e_2,e_1,e_2)&=\lan (A^{\Phi_\th(e_1)}e_1)_{\th_3},(A^{\Phi_\th(e_2)}e_2)_{\th_3}\ran-\lan (A^{\Phi_\th(e_2)}e_1)_{\th_3},(A^{\Phi_\th(e_1)}e_2)_{\th_3}\ran\\
&=\lan A^{\nu_1}e_1,e_3\ran\lan A^{\nu_2}e_2,e_3\ran-\lan A^{\nu_2}e_1,e_3\ran\lan A^{\nu_1}e_2,e_3\ran\\
&=h_{1,13}h_{2,23}-h_{2,13}h_{1,23}=0
\endaligned
\end{equation}
and
\begin{equation}\aligned
U_{\th 0}(e_1,e_2,e_1,e_2)&=\lan (A^{\Phi_\th(e_1)}e_1)_{0},(A^{\Phi_\th(e_2)}e_2)_{0}\ran-\lan (A^{\Phi_\th(e_2)}e_1)_{0},(A^{\Phi_\th(e_1)}e_2)_{0}\ran\\
&=\lan A^{\nu_1}e_1,e_4\ran\lan A^{\nu_2}e_2,e_4\ran-\lan A^{\nu_2}e_1,e_4\ran\lan A^{\nu_1}e_2,e_4\ran\\
&=h_{1,14}h_{2,24}-h_{2,14}h_{1,24}=h_{1,24}^2.
\endaligned
\end{equation}
By (\ref{sec19}),
$$\aligned
0&=\sum_{\si\in \text{Arg}^N,\si\neq \th}\k_{\th\si}R_{\th\si}(e_1,e_2,e_1,e_2)-3\sum_{\si\in \text{Arg}^T,\si\neq \th}
\k_{\th\si}U_{\th\si}(e_1,e_2,e_1,e_2)\\
&=\k_{\th\th_3}R_{\th\th_3}(e_1,e_2,e_1,e_2)-3\k_{\th\th_3}U_{\th\th_3}(e_1,e_2,e_1,e_2)-3\k_{\th 0}U_{\th 0}(e_1,e_2,e_1,e_2)\\
&=(\k_{\th\th_3}-3\k_{\th 0})h_{1,24}^2
\endaligned$$
where
$$\aligned
&\k_{\th\th_3}-3\k_{\th 0}=\f{\sin 2\th}{\cos 2\th-\cos 2\th_3}-\f{3\sin 2\th}{\cos 2\th-1}\\
=&\f{\sin 2\th}{\cos 2\th-\cos 4\th}+\f{3\sin 2\th}{1-\cos 2\th}=\f{2\cos\th\sin\th}{2\sin 3\th \sin \th}+\f{6\cos\th\sin\th}{2\sin^2\th}\\
=&\f{\cos\th(\sin\th+3\sin 3\th)}{\sin 3\th\sin\th}>0.
\endaligned$$
Hence $h_{1,24}=0$ and moreover
\begin{equation}\label{co_sec8}
 h_{3,22}=h_{1,24}=0,\quad h_{3,11}=-h_{2,14}=h_{1,24}=0,\quad h_{3,44}=-h_{1,24}+h_{2,14}=0.
\end{equation}
In conjunction with (\ref{co_sec1}), (\ref{co_sec2}) and (\ref{co_sec3}), we obtain
\begin{equation}\label{co_sec4}
 A^{\nu_3}=0.
\end{equation}

Putting $v=w=e_3$ in (\ref{sec9}) gives
\begin{equation}\label{co_sec5}\aligned
 \lan (\n_{e_3}B)_{e_3e_3},\Phi_{\th_3}(e_3)\ran=&(1/3)\k_{\th_3 \th}\left(\lan B_{e_3 e_3}^\th,B_{e_3 e_3}^\th\ran+2|B_{e_3 e_3}^\th|^2\right)\\
&-2\k_{\th\th_3}\lan B_{e_3e_3}^\th,\Phi_\th(A^{\Phi_{\th_3}(e_3)}e_3)\ran\\
=&\k_{\th_3\th}|B_{e_3e_3}^\th|^2=\k_{\th_3\th}(h_{1,33}^2+h_{2,33}^2)\\
=&\k_{\th_3\th}(h_{2,34}^2+h_{1,34}^2)
\endaligned
\end{equation}
(where we have used (\ref{co_sec4}), (\ref{co6}) and (\ref{co7})). By Lemma \ref{sec31},
\begin{equation}\aligned
 \lan (\n_{e_3}B)_{e_4e_4},\Phi_{\th_3}(e_3)\ran=&-2\k_{\th\th_3}\lan B_{e_3e_4}^\th,\Phi_\th(A^{\Phi_{\th_3}(e_3)}e_4)_\th\ran\\
&+\k_{\th_3 \th}|B_{e_3e_4}^\th|^2+\k_{\th_3 \th}|(A^{\Phi_{\th_3}(e_3)}e_4)_\th|^2+\k_{\th_3 0}|(A^{\Phi_{\th_3}(e_3)}e_4)_0|^2\\
=&\k_{\th_3 \th}(h_{1,34}^2+h_{2,34}^2),
\endaligned
\end{equation}
\begin{equation}\aligned
 \lan (\n_{e_3}B)_{e_1e_1},\Phi_{\th_3}(e_3)\ran=&-2\k_{\th\th_3}\lan B_{e_3e_1}^\th,\Phi_\th(A^{\Phi_{\th_3}(e_3)}e_1)_\th\ran\\
&+\k_{\th_3 \th}|B_{e_3e_1}^\th|^2+\k_{\th_3 \th}|(A^{\Phi_{\th_3}(e_3)}e_1)_\th|^2+\k_{\th_3 0}|(A^{\Phi_{\th_3}(e_3)}e_1)_0|^2\\
=&\k_{\th_3 \th}(h_{1,31}^2+h_{2,31}^2)=0
\endaligned
\end{equation}
and similarly
\begin{equation}\label{co_sec6}
 \lan (\n_{e_3}B)_{e_2e_2},\Phi_{\th_3}(e_3)\ran=0.
\end{equation}
Combining (\ref{co_sec5})-(\ref{co_sec6}) gives
$$\aligned
0&=\lan \n_{e_3}H,\Phi_{\th_3}(e_3)\ran=\sum_{i=1}^4 \lan (\n_{e_3}B)_{e_i e_i},\Phi_{\th_3}(e_3)\ran\\
&=2\k_{\th_3 \th}(h_{1,34}^2+h_{2,34}^2),
\endaligned$$
which forces $h_{1,34}=h_{2,34}=0$ (since $\k_{\th_3 \th}=\f{\sin 2\th_3}{\cos 2\th_3-\cos 2\th}\neq 0$) and moreover
\begin{equation}
 \aligned
h_{1,33}=h_{2,34}=0,\quad &h_{1,44}=-h_{2,34}+h_{3,24}=0;\\
h_{2,33}=-h_{1,34}=0,\quad & h_{2,44}=-h_{3,14}+h_{1,34}=0.
\endaligned
\end{equation}
In conjunction with (\ref{co_sec1}), (\ref{co_sec7}), (\ref{co_sec3}), (\ref{co_sec8}) and (\ref{co_sec4}), we have
$B(p_0)=0$. The arbitrariness of $p_0$ implies $B\equiv 0$, i.e. $M$ is totally geodesic.

\textbf{Case III.} $\th_1=\th_2=\th_3=\pi/3$. Then $g^N=1$, $g^T=2$ and Theorem \ref{tri} implies $M$ is affine linear.

\textbf{Case IV.} $\th_2=\th_3\in (\pi/3,\arccos (\sqrt{6}/6))\cup (\arccos (\sqrt{6}/6),\pi/2)$ and $\th_1=\pi-2\th_2$.
Denote $\th:=\th_2$, then $\text{Arg}^N=\{\th,\th_1\}$, $\text{Arg}^T=\{0,\th,\th_1\}$;
$T_{p_0,\th}M=\text{span}\{e_2,e_3\}$ and $N_{p_0,\th}M=\text{span}\{\nu_2,\nu_3\}$ with $\nu_\a=\Phi_\th(e_\a)$
for each $2\leq \a\leq 3$; $T_{p_0,\th_1}M=\text{span}\{e_1\}$ and $N_{p_0,\th_1}M=\text{span}\{\nu_1\}$
with $\nu_1=\Phi_{\th_1}(e_1)$; $T_{p_0,0}M=\text{span}\{e_4\}$. Applying Lemma \ref{sec2} and \ref{sec6} gives
\begin{equation}\label{co_sec12}
 h_{2,33}=h_{3,22}=0,\quad 0=h_{2,31}+h_{3,21}=h_{2,34}+h_{3,24}.
\end{equation}
Substituting the above equations into (\ref{co5})-(\ref{co8}) yields
\begin{eqnarray}
 &&h_{1,23}=h_{2,31}=h_{3,12}=0;\label{co_sec17}\\
&&h_{1,22}=-h_{3,24}=h_{2,34}=h_{1,33},\quad h_{1,44}=-2h_{2,34};\\
&&h_{1,34}=0,\quad h_{2,11}=h_{3,14}=-h_{2,44};\label{co_sec9}\\
&&h_{1,24}=0,\quad h_{3,11}=-h_{2,14}=-h_{3,44}.\label{co_sec10}
\end{eqnarray}
A straightforward calculation shows
\begin{equation}\aligned
 R_{\th\th_1}(e_2,e_3,e_2,e_3)&=\lan B_{e_2e_2}^{\th_1},B_{e_3e_3}^{\th_1}\ran-\lan B_{e_2e_3}^{\th_1},B_{e_3e_2}^{\th_1}\ran\\
&=h_{1,22}h_{1,33}-h_{1,23}h_{1,32}=h_{2,34}^2,
\endaligned
\end{equation}
\begin{equation}\aligned
 U_{\th\th_1}(e_2,e_3,e_2,e_3)&=\lan (A^{\Phi_\th(e_2)}e_2)_{\th_1},A^{\Phi_\th(e_3)}e_3)_{\th_1}\ran-
\lan A^{\Phi_\th(e_3)}e_2)_{\th_1},A^{\Phi_\th(e_2)}e_3)_{\th_1}\ran\\
&=h_{2,21}h_{3,31}-h_{3,21}h_{2,31}=0,
\endaligned
\end{equation}
\begin{equation}\aligned
 U_{\th 0}(e_2,e_3,e_2,e_3)&=\lan A^{\Phi_\th(e_2)}e_2)_{0}, A^{\Phi_\th(e_3)}e_3)_{0}\ran-\lan  A^{\Phi_\th(e_3)}e_2)_{0},
 A^{\Phi_\th(e_2)}e_3)_{0}\ran\\
&=h_{2,24}h_{3,34}-h_{3,24}h_{2,34}=h_{2,34}^2,
\endaligned
\end{equation}
and then Lemma \ref{sec30} implies
\begin{equation}\aligned
 0&=\sum_{\si\in \text{Arg}^N,\si\neq 0} \k_{\th\si}R_{\th\si}(e_2,e_3,e_2,e_3)-3\sum_{\si\in \text{Arg}^T,\si\neq \th}
\k_{\th\si}U_{\th\si}(e_2,e_3,e_2,e_3)\\
&=\k_{\th\th_1}R_{\th\th_1}(e_2,e_3,e_2,e_3)-3\k_{\th\th_1}U_{\th\th_1}(e_2,e_3,e_2,e_3)-3\k_{\th 0}U_{\th 0}(e_2,e_3,e_2,e_3)\\
&=(\k_{\th\th_1}-3\k_{\th 0})h_{2,34}^2,
\endaligned
\end{equation}
where
\begin{equation}
 \aligned
&\k_{\th\th_1}-3\k_{\th 0}=\f{\sin 2\th}{\cos 2\th-\cos 2\th_1}-\f{3\sin 2\th}{\cos 2\th-1}\\
=&\f{\sin 2\th}{\cos 2\th-\cos 4\th}+\f{3\sin 2\th}{1-\cos 2\th}=\f{\cos\th(\sin\th+3\sin 3\th)}{\sin 3\th\sin \th}\\
=&\f{2\cos\th(5-6\sin^2 \th)}{\sin 3\th}.
\endaligned
\end{equation}
Since $\th\neq \arccos (\sqrt{6}/6)$, $\sin^2\th\neq 5/6$ and then $\k_{\th\th_1}-3\k_{\th 0}\neq 0$. Hence
$h_{2,34}=0$ and moreover
\begin{equation}
 h_{1,22}=h_{1,33}=h_{1,44}=h_{3,24}=h_{2,34}=0.
\end{equation}
In conjunction with (\ref{co_sec1}), (\ref{co_sec17}), (\ref{co_sec9}) and (\ref{co_sec10}),
we have $A^{\nu_1}=0$. With the aid of Lemma \ref{sec31}, one can then proceed as in Case II to deduce that
$B(p_0)=0$. Since $p_0$ is arbitrary, $M$ has to be affine linear.

\end{proof}

\begin{pro}\label{p2}
Let $M$ be a coassociative submanifold of $\Im \O$. Assume $M$ has CJA relative to $\Im \H$, and $\text{Arg}^N=\{\arccos (\sqrt{6}/6),\arccos(2/3)\}$,
then either $M$ is affine linear, or there exists $a_0\in \text{Sp}_1$, such that $M$ is a translate of an open subset of
$M(a_0)$. Here $M(a_0)$ denotes the Lawson-Osserman's cone, see (\ref{l-o-cone}).
\end{pro}

\begin{proof}
Let $\th_1:=\arccos(2/3)$, $\th_2=\th_3:=\arccos(\sqrt{6}/6)$
and $\th:=\th_2$. For an arbitrary point $p_0\in M$, let $\{e_1,e_2,e_3,e_4\}$ be an orthonormal tangent frame field and $\{\nu_1,\nu_2,\nu_3\}$
be an orthonormal normal frame field on $U$, a neighborhood of $p_0$, such that for any $p\in M$, $e_i(p)$ and $\nu_\a(p)$ satisfy the properties
in Proposition \ref{coass}. In particular, $\nu_\a=\Phi_\a(e_\a)$
for each $1\leq \a\leq 3$.
 With the aid of Lemma \ref{sec2}, Lemma \ref{sec6} and Proposition \ref{co_sec11}, one can proceed as
above to get some pointwise relations between the coefficients of the second fundamental form, see (\ref{co_sec1}), (\ref{co_sec12})-(\ref{co_sec10}). Denote
\begin{equation}\label{h}
h:=h_{1,22},
\end{equation}
then $h$ can be seen as a smooth function on $U$, and
\begin{equation}
h_{1,33}=h_{2,34}=h,\quad h_{3,24}=-h,\quad h_{1,44}=-2h.
\end{equation}

\textbf{Step I.} Prove
\begin{equation}\label{co_sec15}
h_{2,11}=h_{3,11}=h_{2,14}=h_{3,14}=h_{2,44}=h_{3,44}=0.
\end{equation}

By (\ref{sec9}),
\begin{equation}\label{co_sec13}
\aligned
\lan (\n_{e_1}B)_{e_1e_1},\Phi_{\th_1}(e_1)\ran=&(1/3)\k_{\th_1\th}\left(\lan B_{e_1e_1}^\th,B_{e_1e_1}^\th\ran+2|B_{e_1e_1}^\th|^2\right)\\
&-2\k_{\th\th_1}\lan B_{e_1e_1}^\th,\Phi_\th(A^{\Phi_{\th_1}(e_1)}e_1)\ran\\
=&\k_{\th_1\th}(h_{2,11}^2+h_{3,11}^2).
\endaligned
\end{equation}

Applying Lemma \ref{sec31}, we have
\begin{equation}\aligned
\lan (\n_{e_1}B)_{e_4e_4},\Phi_{\th_1}(e_1)\ran=&-2\k_{\th\th_1}\lan B_{e_1e_4}^\th,\Phi_\th(A^{\Phi_{\th_1}}e_4)_\th\ran\\
&+\k_{\th_1\th}|B_{e_1e_4}^\th|^2+\k_{\th_1\th}|(A^{\Phi_{\th_1}(e_1)}e_4)_\th|^2+\k_{\th_1 0}|(A^{\Phi_{\th_1}(e_1)}e_4)_0|^2\\
=&-2\k_{\th\th_1}(h_{2,14}h_{1,42}+h_{3,14}h_{1,43})\\
&+\k_{\th_1\th}(h_{2,14}^2+h_{3,14}^2)+\k_{\th_1\th}(h_{1,42}^2+h_{1,43}^2)+\k_{\th_1 0}h_{1,44}^2\\
=&\k_{\th_1\th}(h_{2,14}^2+h_{3,14}^2)+4\k_{\th_1 0}h^2,
\endaligned
\end{equation}
\begin{equation}\aligned
\lan (\n_{e_1}B)_{e_2e_2},\Phi_{\th_1}(e_1)\ran=&-2\k_{\th\th_1}\lan B_{e_1e_2}^\th,\Phi_\th(A^{\Phi_{\th_1}}e_2)_\th\ran\\
&+\k_{\th_1\th}|B_{e_1e_2}^\th|^2+\k_{\th_1\th}|(A^{\Phi_{\th_1}(e_1)}e_2)_\th|^2+\k_{\th_1 0}|(A^{\Phi_{\th_1}(e_1)}e_2)_0|^2\\
=&-2\k_{\th\th_1}(h_{2,12}h_{1,22}+h_{3,12}h_{1,23})\\
&+\k_{\th_1\th}(h_{2,12}^2+h_{3,12}^2)+\k_{\th_1\th}(h_{1,22}^2+h_{1,23}^2)+\k_{\th_1 0}h_{1,24}^2\\
=&\k_{\th_1\th}h_{1,22}^2=\k_{\th_1\th}h^2
\endaligned
\end{equation}
and similarly
\begin{equation}\label{co_sec14}
\lan (\n_{e_1}B)_{e_3e_3},\Phi_{\th_1}(e_1)\ran=\k_{\th_1 \th}h_{1,33}^2=\k_{\th_1\th}h^2.
\end{equation}
Adding (\ref{co_sec13})-(\ref{co_sec14}) gives
$$\aligned
0&=\lan \n_{e_1}H,\Phi_{\th_1}(e_1)\ran=\sum_{i=1}^4 \lan (\n_{e_1}B)_{e_ie_i},\Phi_{\th_1}(e_1)\ran\\
&=\k_{\th_1\th}(h_{2,11}^2+h_{3,11}^2+h_{2,14}^2+h_{3,14}^2)+2(2\k_{\th_1 0}+\k_{\th_1\th})h^2,
\endaligned$$
where
\begin{equation}\label{k10}\aligned
&2\k_{\th_1 0}+\k_{\th_1 \th}=\f{2\sin 2\th_1}{\cos 2\th_1-1}+\f{\sin 2\th_1}{\cos 2\th_1-\cos 2\th}\\
=&\f{2\sin 2\th_1}{\cos 2\th_1-1}+\f{\sin 2\th_1}{\cos 2\th_1+\cos\th_1}=\f{\sin 2\th_1\big[2(\cos 2\th_1+\cos\th_1)+\cos 2\th_1-1\big]}{(\cos 2\th_1-1)(\cos 2\th_1+\cos\th_1)}\\
=&\f{2\sin 2\th_1(3\cos \th_1-2)(\cos\th_1+1)}{(\cos 2\th_1-1)(\cos 2\th_1+\cos\th_1)}=0.
\endaligned
\end{equation}
Hence $h_{2,11}=h_{3,11}=h_{2,14}=h_{3,14}=0$ and substituting it into (\ref{co_sec9})-(\ref{co_sec10}) implies $h_{2,44}=h_{3,44}=0$.

\textbf{Step II.} Calculation of the connection coefficients.

Denote
\begin{equation}
 \G_{ij}^k:=\lan \n_{e_i}e_j,e_k\ran,\quad \bar{\G}_{i\a}^\be:=\lan \n_{e_i}\nu_\a,\nu_\be\ran.
\end{equation}
Then differentiating both sides of $\lan e_i,e_k\ran=\de_{jk}$ with respect to $e_i$ gives
$\G_{ij}^k+\G_{ik}^j=0$. In particular, $\G_{ij}^j=0$. Similarly
$\bar{\G}_{i\a}^\be+\bar{\G}_{i\be}^\a=0$ and especially $\bar{\G}_{i\a}^\a=0$.

Based on Lemma \ref{sec5}, a direct calculation shows

\begin{equation}\aligned
 \G_{i1}^4&=(S_{\th_1 0})_{e_1e_4}(e_i)=\k_{0\th_1}\lan B_{e_i e_1},\Phi_0(e_4)\ran-\k_{\th_1 0}\lan B_{e_i e_4},\Phi_{\th_1}(e_1)\ran\\
&=-\k_{\th_1 0}h_{1,i4},
\endaligned
\end{equation}
\begin{equation}\aligned
 \G_{i1}^2&=(S_{\th_1 \th})_{e_1e_2}(e_i)=\k_{\th\th_1}\lan B_{e_i e_1},\Phi_\th(e_2)\ran-\k_{\th_1 \th}\lan B_{e_i e_2},\Phi_{\th_1}(e_1)\ran\\
&=\k_{\th\th_1}h_{2,i1}-\k_{\th_1 \th}h_{1,i2},
\endaligned
\end{equation}
\begin{equation}\aligned
 \G_{i2}^4&=(S_{\th 0})_{e_2e_4}(e_i)=\k_{0\th}\lan B_{e_i e_2},\Phi_0(e_4)\ran-\k_{\th 0}\lan B_{e_i e_4},\Phi_{\th}(e_2)\ran\\
&=-\k_{\th 0}h_{2,i4}
\endaligned
\end{equation}
and similarly
\begin{equation}
 \G_{i1}^3=(S_{\th_1 \th})_{e_1 e_3}(e_i)=\k_{\th\th_1}h_{3,i1}-\k_{\th_1 \th}h_{1,i3},
\end{equation}
\begin{equation}
 \G_{i3}^4=(S_{\th 0})_{e_3e_4}(e_i)=-\k_{\th 0}h_{3,i4}.
\end{equation}
By Lemma \ref{sec7},
\begin{equation}\aligned
 \bar{\G}_{21}^3&=(S_{\th_1 \th}^N)_{\nu_1 \nu_3}(e_2)= \k_{\th_1\th}
\lan B_{e_2,\Phi_{\th_1}(\nu_1)},\nu_3\ran-\k_{\th\th_1}\lan B_{e_2,\Phi_\th(\nu_3)},\nu_1\ran\\
&=\k_{\th\th_1}h_{1,23}-\k_{\th_1\th}h_{3,21}=0.
\endaligned
\end{equation}

\textbf{Step III.} Proof that the angle lines with respect to $\th_1$, i.e. integral curves of the vector field $e_1$,
 must be straight lines in Euclidean space.

This is equivalent to  $\ol{\n}_{e_1}e_1=0$ holding everywhere, which follows from the  following straightforward calculation.
$$\aligned
\ol{\n}_{e_1}e_1&=B_{e_1e_1}+\n_{e_1}e_1=h_{\a,11}\nu_\a+ \G_{11}^i e_i=\sum_{i=2}^3 \G_{11}^i e_i+\G_{11}^4 e_4\\
&=\sum_{i=2}^3 (\k_{\th\th_1}h_{i,11}-\k_{\th_1 \th}h_{1,1i})e_i-\k_{\th_1 0}h_{1,14}e_4=0.
\endaligned$$

\textbf{Step IV.} Proof that there exists a hypersurface $N$ of $U$, such that $p_0\in N$ and $e_1(p)\bot T_p N$ for every $p\in N$.

By the  Frobenius theorem, it suffices to prove that the subbundle $e_1^\bot$ of $TU$ is integrable; more precisely, given arbitrary smooth sections $X,Y$ of
$e_1^\bot$, $[X,Y]$ takes values in $e_1^\bot$ as well.

Now we write $X=\sum_{i=2}^4 X^i e_i$ and $Y=\sum_{j=2}^4 Y^j e_j$, then
$$[X,Y]=X^i Y^j [e_i,e_j]+X^i (\n_{e_i}Y^j)e_j-Y^j(\n_{e_j}X^i)e_i$$
and hence
$$\lan [X,Y],e_1\ran=X^i Y^j\lan [e_i,e_j],e_1\ran.$$
Hence it is necessary and sufficient for us to show $\lan [e_i,e_j],e_1\ran=0$ for any $2\leq i<j\leq 4$.

Since $\n$ is torsion-free,
$$\aligned
\lan [e_2,e_3],e_1\ran=&\lan \n_{e_2}e_3,e_1\ran-\lan \n_{e_3}e_2,e_1\ran=\G_{23}^1-\G_{32}^1=-\G_{21}^3+\G_{31}^2\\
=&-(\k_{\th\th_1}h_{3,21}-\k_{\th_1\th}h_{1,23})+(\k_{\th\th_1}h_{2,31}-\k_{\th_1\th}h_{1,32})\\
=&0,
\endaligned$$
$$\aligned
\lan [e_2,e_4],e_1\ran=&\lan \n_{e_2}e_4,e_1\ran-\lan \n_{e_4}e_2,e_1\ran=\G_{24}^1-\G_{42}^1=-\G_{21}^4+\G_{41}^2\\
=&\k_{\th_1 0}h_{1,24}+(\k_{\th\th_1}h_{2,41}-\k_{\th_1\th}h_{1,42})\\
=&0
\endaligned$$
and similarly
$$\lan [e_3,e_4],e_1\ran=\k_{\th_1 0}h_{1,34}+(\k_{\th\th_1}h_{3,41}-\k_{\th_1 \th}h_{1,43})=0.$$
Then the claim is proved.

Without loss of generality, we can assume that the closure of $N$ is contained in $U$. Then
there exists $\de>0$, such that $\mb{X}(p)+te_1\in U$ for every $p\in N$ and any $t\in (-\de,\de)$, where
$\mb{X}(p)$ denotes the position vector of $p$ in $\Im \O$.
Define $\phi: N\times (-\de,\de)\ra U$
\begin{equation}
 (p,t)\mapsto \mb{X}(p)+te_1,
\end{equation}
then $\phi$ is a diffeomorphism between $N\times (-\de,\de)$ and a neighborhood of $p_0$ in $M$, which is denoted by
$W$.

\textbf{Step V.} The function $h$ defined in (\ref{h}) is constant on $N$.

Applying the Codazzi equations,
$$\aligned
\n_{e_4}h&=\n_{e_4}h_{1,22}=\n_{e_4}\lan B_{e_2e_2},\nu_1\ran\\
&=\lan (\n_{e_4}B)_{e_2e_2},\nu_1\ran+2\G_{42}^i h_{1,2i}+\G_{41}^\a h_{\a,22}\\
&=\lan (\n_{e_4}B)_{e_2e_2},\nu_1\ran=\lan (\n_{e_2}B)_{e_2e_4},\nu_1\ran\\
&=\n_{e_2}h_{1,24}-\G_{22}^i h_{1,i4}-\G_{24}^i h_{1,2i}-\bar{\G}_{21}^\a h_{\a,24}\\
&=2\G_{22}^4 h-\G_{24}^2 h+\bar{\G}_{21}^3 h=3\G_{22}^4 h\\
&=-3\k_{\th 0}h_{2,24}h=0,
\endaligned$$
$$\aligned
\n_{e_2}h&=\n_{e_2}h_{1,33}=\n_{e_2}\lan B_{e_3e_3},\nu_1\ran\\
&=\lan (\n_{e_2}B)_{e_3e_3},\nu_1\ran+2\G_{23}^i h_{1,3i}+\bar{\G}_{21}^\a h_{\a,33}\\
&=\lan (\n_{e_2}B)_{e_3e_3},\nu_1\ran=\lan (\n_{e_3}B)_{e_2e_3},\nu_1\ran\\
&=\n_{e_3}h_{1,23}-\G_{32}^i h_{1,i3}-\G_{33}^i h_{1,2i}-\bar{\G}_{31}^\a h_{\a,23}\\
&=-\G_{32}^3 h-\G_{33}^2 h=0
\endaligned$$
and similarly
$$\n_{e_3}h=\n_{e_3}h_{1,22}=-\G_{23}^2 h-\G_{22}^3 h=0.$$
Hence $\n h\equiv 0$ on $N$.
Without loss of generality, we can assume $h|_N\equiv h_0$, with $h_0$ a nonnegative constant.

\textbf{Step VI.}  $W$ is a cone whenever $h_0>0$.

Define $\psi:N\ra \Im\O$
$$\psi(p)=\mb{X}(p)+R_0e_1(p),$$
where $R_0$ is a constant to be chosen. Then
$$\aligned
\psi_* e_i&=e_i+R_0\ol{\n}_{e_i}e_1\\
&=e_i+R_0\big(B_{e_i e_1}+\n_{e_i}e_1\big)\\
&=e_i+R_0\G_{i1}^j e_j\\
&=e_i+R_0\big[\sum_{j=2}^3 (\k_{\th\th_1}h_{j,i1}-\k_{\th_1\th}h_{1,ij})e_j-\k_{\th_1 0}h_{1,i4}e_4\big]
\endaligned$$
for each $2\leq i\leq 4$. More precisely,
\begin{equation}\label{co_sec16}\aligned
\psi_* e_2&=(1-R_0\k_{\th_1 \th}h_{1,22})e_2=(1-R_0 \k_{\th_1\th}h_0)e_2,\\
\psi_* e_3&=(1-R_0\k_{\th_1 \th}h_{1,33})e_3=(1-R_0\k_{\th_1\th}h_0)e_3,\\
\psi_* e_4&=(1-R_0\k_{\th_1 0}h_{1,44})e_4=(1+2R_0\k_{\th_1 0}h_0)e_4.
\endaligned
\end{equation}
Now we put
$$R_0:=(\k_{\th_1\th}h_0)^{-1},$$
then combining (\ref{co_sec16}) and (\ref{k10}) implies $\psi_* e_i=0$ for each $2\leq i\leq 4$. Hence
$\psi$ is a constant map on $N$. Without loss of generality, we can assume $\psi\equiv 0$,
i.e. $F(p)=-R_0e_1(p)$ for every $p\in N$. In other words, $N$ lies in the Euclidean sphere centered at
$0$ and of radius $R_0$, and an arbitrary normal line of $N$, i.e. $\{F(p)+te_1:t\in \R\}$ with
$p\in N$, must go through the origin. Therefore $W$ is a cone.

\textbf{Step VII.}  $M$ is an open subset of $M(a_0)$ provided that $h_0>0$ and $\psi\equiv 0$.

Let
\begin{equation}
 S:=\{x\in \Im\O:|x|=R_0,|\mc{P}_0^\bot x|=\cos\th_1 R_0\}
\end{equation}
be a submanifold of $\Im\O$. For any $x\in S$, there exist a unit element $b\in \Im\H$ and a unit element
$\ep\in \H e$, such that
$$x=R_0(-\sin\th_1 b+\cos\th_1 b\ep).$$
Define
\begin{equation}
 E_x=\R \ep\oplus \{\sin\th c-\cos\th c\ep: c\in \Im\H, \lan b,c\ran=0\},
\end{equation}
then $E_x$ is a $3$-dimensional subspace of $T_x S$. Furthermore
\begin{equation}
 E:=\{E_x:x\in S\}
\end{equation}
is a $3$-dimensional distribution on $S$.

For any $p\in N$, $e_1(p)$ is a unit tangent angle direction associated to $\th_1$. Hence there exist $b\in \Im\H$
and $\ep\in \H e$ satisfying $|b|=|\ep|=1$, such that
$$e_1(p)=\sin\th_1 b-\cos\th_1 b\ep.$$
Moreover,
$$\aligned
\mb{X}(p)&=\psi(p)-R_0 e_1(p)=-R_0 e_1(p)\\
&=R_0(-\sin\th_1 b+\cos\th_1 b\ep).
\endaligned$$
Therefore $N\subset S$.

Denote
$$\aligned
\nu_1:=&\Phi_{\th_1}(e_1)=(-\tan\th_1 \mc{P}_0^\bot+\cot\th_1 \mc{P}_0)e_1\\
=&\cos\th_1 b+\sin\th_1 b\ep,
\endaligned$$
then $\nu_1$ is a unit angle direction of $N_p M$ with respect to $\Im \H$. On the other hand,
Proposition \ref{o8} implies the existence of an orthonormal basis $\{b_1,b_2,b_3\}$ of $\Im\H$
satisfying $b_3=b_1 b_2$ and a unit element $\ep'\in \H e$, such that
$$\nu'_\a:=\cos\th_\a b_\a+\sin\th_\a b_\a \ep'\qquad \forall 1\leq \a\leq 3$$
are all unit angle directions of $N_p M$ relative to $\Im\H$. Since $m_{\th_1}=1$,
$\nu'_1=\pm \nu_1$, and then one can assume $b_1=b$, $\ep'=\ep$ without loss of generality, which implies
$$N_p M=\R \nu_1\oplus \{\cos\th c+\sin\th c\ep: c\in \Im H,\lan b,c\ran=0\}.$$
Noting that $T_p N\bot N_p M$ and $T_p N\bot e_1$, it is easy to deduce that
$T_p N=E_p$, i.e. $N$ is an integral manifold of $E$.

For any $a\in \text{Sp}_1$, $M(a)$ is a coassociative cone, which has CJA with $\text{Arg}^T=\{\th_1,\th,0\}$,
and each ray is an angle line with respect to $\th_1$. As above, one can show that $M(a)\cap B(R_0)\subset S$
and that it is also an integral manifold of $E$.

Now we write
\begin{equation}
 \mb{X}(p_0)=R_0(\sin\th_1 b_0+\cos\th_1 c_0 e)=(2/3)R_0\big[(\sqrt{5}/2)b_0+c_0 e\big]
\end{equation}
with $b_0\in \Im \H, c_0\in \H$ satisfying $|b_0|=|c_0|=1$. Then choosing
\begin{equation}
 a_0:=c_0b_0\bar{c}_0,\quad q_0:=\bar{c}_0
\end{equation}
gives
$$\mb{X}(p_0)=(2/3)R_0\big[(\sqrt{5}/2)q_0 a_0 \bar{q}_0+\bar{q}_0 e\big]\in M(a_0).$$
Therefore $N$ and $M(a_0)\cap B(R_0)$ are both integral manifolds of $E$. Since $M(a_0)\cap B(R_0)$ is complete,
applying the Frobenius theorem implies $N\subset M(a_0)\cap B(R_0)$, and hence $W\subset M(a_0)$. Finally, because minimal submanfolds
in Euclidean space are analytic manifolds, $M$ has to be an open subset of $M(a_0)$.

\textbf{Step VIII.} $M$ is affine linear whenever $h_0=0$.

Firstly, $h_0=0$ implies $B\equiv 0$ on $N$. Denote by $\td{B}$ the second fundamental form of $N$ in $\Im\O$, then
$$\lan \td{B}_{e_i e_j},\nu_\a\ran=\lan \ol{\n}_{e_i}e_j,\nu_\a\ran=\lan B_{e_i e_j},\nu_\a\ran=0$$
for any $2\leq i,j\leq 4$ and $1\leq \a\leq 3$,
$$\aligned
\lan \td{B}_{e_ie_j},e_1\ran&=\lan \ol{\n}_{e_i}e_j,e_1\ran=\lan \n_{e_i}e_j,e_1\ran=\G_{ij}^1\\
&=-\G_{i1}^j=-(\k_{\th\th_1}h_{j,i1}-\k_{\th_1\th}h_{1,ij})=0
\endaligned$$
for any $2\leq j\leq 3$ and
$$\lan \td{B}_{e_ie_4},e_1\ran=\G_{i4}^1=-\G_{i1}^4=\k_{\th_1 0}h_{1,i4}=0.$$
Thus $\td{B}\equiv 0$, i.e. $N$ is totally geodesic.

Since
$$\aligned
\ol{\n}_{e_i}e_1&=\sum_{j=2}^4 \lan \n_{e_i}e_1,e_j\ran e_j+B_{e_ie_1}\\
&=\sum_{j=2}^4 \G_{i1}^j e_j+B_{e_ie_1}=0
\endaligned$$
for each $2\leq i\leq 4$, $e_1$ is parallel along $N$. Therefore $W$ is an open subset of an affine
linear subspace of $\Im \O$. Due to the analyticity of minimal submanifolds, $M$ has to be affine linear. And
the proof is completed.

\end{proof}

Proposition \ref{p1} and Proposition \ref{p2} together imply the following theorem.

\begin{thm}
 Let $M$ be a coassociative submanifold in $\Im \O$. Assume $M$ has CJA relative to $\Im \H$. If
$g^N\leq 2$ and $\pi/2\notin \text{Arg}^N$, then either $M$ is affine linear, or there exists $a_0\in \text{Sp}_1$ and $w_0\in \Im \O$, such that $M$ is an open subset of
\begin{equation*}
M(a_0,w_0):=\{r\big[(\sqrt{5}/2) qa_0 \bar{q}+\bar{q}e\big]+w_0:q\in \text{Sp}_1,r\in \Bbb{R}^+\}.
\end{equation*}
In other words, $M$ is a translate of a portion of the Lawson-Osserman's cone.
\end{thm}

As at the end of Section \ref{s4}, we have a corollary.

\begin{cor}
Let $D$ be an open domain of $\H$ and $f: D\ra \Im \H$. If $M=\text{graph }f$ is
a coassociative submanifold with CJA relative to $\Im \H$, and $g^N\leq 2$,
then $f$ is either an affine linear function or
$f(x)=\eta(x-x_0)+y_0$,
where $x_0\in \H$, $y_0\in \Im \H$ and
$$\eta(x)=\f{\sqrt{5}}{2|x|}\bar{x}\ep x$$
with $\ep$ an arbitrary unit element in $\Im \H$.

\end{cor}

This is the Theorem 1.2 in \S 1.5.

\bibliographystyle{amsplain}

\end{document}